\documentclass[sn-mathphys]{sn-jnl}% Math and Physical Sciences Reference Style
\usepackage{amsmath}
\usepackage{amsfonts}
\usepackage{amssymb}
\usepackage{amsbsy}
\usepackage{subfigure} 
\usepackage{graphicx}
\usepackage{epstopdf}
\usepackage{breakurl}
\usepackage{subfigure} 
\usepackage{algpseudocode}  

\usepackage{bbm}       								  %空心数字1
\usepackage{longtable}
\usepackage{bm}    			
\usepackage{color}
\usepackage{enumerate}
\usepackage{float}             					 		%强制图片位置			
\usepackage{graphicx}
\usepackage{listings} 								
\usepackage{mathrsfs}							   %花体字
\usepackage{rotating}   				\usepackage{lineno}
		
\usepackage{placeins}				%禁止浮动跨过section
\usepackage{stfloats}
\usepackage{pdflscape}
\usepackage{afterpage}
\usepackage{capt-of}% or 

\usepackage{ulem}

\makeatletter
\newif\if@restonecol
\makeatother

\usepackage[linesnumbered,ruled,vlined]{algorithm2e}%[ruled,vlined]{
\jyear{2021}%
\usepackage{essay-def-siam}

\numberwithin{equation}{section}

\graphicspath{{./figures/}}            %  设置加载图片的路径

 \newcommand{\blue}[1]{{\color{blue}#1}}
\theoremstyle{thmstyleone}%
\newtheorem{lemma}{Lemma}
\newtheorem{assumption}{Assumption}
\newtheorem{corollary}{Corollary}
\newtheorem{theorem}{Theorem}%  meant for continuous numbers
\newtheorem{proposition}[theorem]{Proposition}% 

\theoremstyle{thmstyletwo}%
\newtheorem{remark}{Remark}%

\theoremstyle{thmstylethree}%
\newtheorem{definition}{Definition}%

\begin{document}
%\linenumber
\title[Riemannian Smoothing Gradient Type Algorithms]{Riemannian Smoothing  Gradient Type Algorithms for Nonsmooth Optimization Problem on Compact Riemannian Submanifold Embedded in Euclidean Space}

%%=============================================================%%
%% Prefix	-> \pfx{Dr}
%% GivenName	-> \fnm{Joergen W.}
%% Particle	-> \spfx{van der} -> surname prefix
%% FamilyName	-> \sur{Ploeg}
%% Suffix	-> \sfx{IV}
%% NatureName	-> \tanm{Poet Laureate} -> Title after name
%% Degrees	-> \dgr{MSc, PhD}
%% \author*[1,2]{\pfx{Dr} \fnm{Joergen W.} \spfx{van der} \sur{Ploeg} \sfx{IV} \tanm{Poet Laureate} 
%%                 \dgr{MSc, PhD}}\email{iauthor@gmail.com}
%%=============================================================%%

\author[1]{\fnm{Zheng} \sur{Peng}}%\email{pzheng@xtu.edu.cn}

\author[1]{\fnm{Weihe} \sur{Wu}}%\email{1009248350@qq.com}
%\equalcont{These authors contributed equally to this work.}

\author[2]{\fnm{Jiang} \sur{Hu}}%\email{hujiangopt@gmail.com}
%\equalcont{These authors contributed equally to this work.}

\author*[3]{\fnm{Kangkang} \sur{Deng}}\email{freedeng1208@gmail.com}

\affil[1]{\orgdiv{School of Mathematics and Computational Science}, \orgname{Xiangtan University}, \orgaddress{ \city{Xiangtan}, \postcode{411105}, \country{China}}}

\affil[2]{\orgdiv{Massachusetts General Hospital and Harvard Medical School}, \orgname{Harvard University}, \orgaddress{\city{Boston}, \postcode{02114}, \country{United States}}}

\affil*[3]{\orgdiv{Department of Mathematics}, \orgname{National University of Defense Technology}, \orgaddress{ \city{Changsha}, \postcode{420000},  \country{China}}}

%%==================================%%
%% sample for unstructured abstract %%
%%==================================%%

\abstract{In this paper, we introduce the notion of generalized $\epsilon$-stationarity for a class of nonconvex and nonsmooth composite minimization problems on compact Riemannian submanifold embedded in  Euclidean space. 
To find a generalized $\epsilon$-stationarity point,   we develop a family of Riemannian gradient-type methods based on the Moreau envelope technique with a decreasing sequence of smoothing parameters,  namely Riemannian smoothing gradient and Riemannian smoothing stochastic gradient methods.  We prove that the Riemannian smoothing gradient method has the iteration complexity of $\mathcal{O}(\epsilon^{-3})$ for driving a generalized $\epsilon$-stationary point. To our knowledge, this is the best-known iteration complexity result for the nonconvex and nonsmooth composite problem on manifolds.  For the  Riemannian smoothing stochastic gradient method, one can achieve the iteration complexity of $\mathcal{O}(\epsilon^{-5})$ for driving a generalized $\epsilon$-stationary point. Numerical experiments are conducted to validate the superiority of our algorithms.}

% \noindent\textbf{Mathematics Subject Classification:  } 	
	
%%\pacs[JEL Classification]{D8, H51}


\maketitle

\section{Introduction}
We consider a general nonsmooth optimization problem on manifolds:
\begin{equation}\label{prob:p1}
\min_{x\in\mcM}  \quad  \varphi(x):= f(x)+h(\mcA x),
\end{equation}
where $\mcM$ is a compact Riemannian manifold embedded in $\mbR^n$, $f:\mcM\rightarrow\mbR$ is a continuously-differentiable nonconvex function; $h:\mbR^m \rightarrow\mbR$ is a nonsmooth weakly-convex function with a readily available proximal operator, and $\mcA:\mbR^n \rightarrow \mbR^m$ is a linear operator.
 The model \eqref{prob:p1} includes a wide range of fundamental problems such as sparse principal component analysis (SPCA) \cite{jolliffe2003modified}, range-based independent component analysis  \cite{selvan2013spherical,selvan2015range}, distributed optimization \cite{deng2023decentralized} and robust low-rank matrix completion \cite{cambier2016robust,hosseini2017riemannian}, see 
Absil et al. \cite{absil2017collection} for more examples. In those applications, $h$ is often a sparse regularization function, and problem \eqref{prob:p1} becomes a sparse optimization problem on Riemannian manifolds. %In addition, when $h$ is an indicator function onto a set $\Omega$, i.e., $g(x) = \delta_{\Omega}(x)$, problem \eqref{prob:p1} represents a optimization problem on Riemannian manifolds with constraints \cite{metel2019simple}:
% \begin{equation}
%   \begin{split}
%      \min_{x\in\mathcal{M}}   f(x), ~~\mbox{s.t. } &  x\in\Omega.
%   \end{split}
% \end{equation}
% It has many appealing applications, such as non-negative principal component analysis \cite{montanari2015non} and minimum balanced cut for graph bisection \cite{lang2006fixing} etc.

 In many applications of model \eqref{prob:p1},    $\mathcal{M}$ is a submanifold embedded in Euclidean space,  for example,  Stiefel manifold. There has a connection in function properties over Riemannian manifolds and Euclidean space, including smoothness and convexity. So it allows us to extend some classical optimization methods in Euclidean space to  Riemannian manifolds. Indeed, many classical optimization methods for  smooth optimization problems  in Euclidean space have been successfully generalized to the problems on Riemannian manifolds, e.g.,  gradient-based methods \cite{Luenberger1972The,grocf}, Newton-based methods \cite{Absil2007Trust,HuaAbsGal2018,hu2018adaptive}, splitting based methods \cite{Kovna2015,deng2022manifold} and stochastic gradient methods \cite{bonnabel2013stochastic,sato2019riemannian,shah2021stochastic}.  The reader is referred to  \cite{boumal2023intromanifolds,sato2021riemannian,hu2020brief} for a comprehensive review.  However, these methods are not adapted to problem \eqref{prob:p1} since there exists a nonsmooth term in the objective. In this paper, we propose a homotopy smoothing method to solve problem \eqref{prob:p1} in some general settings. Firstly, we utilize the Moreau envelope technique to obtain a smooth approximation of nonsmooth function $h$ and employ Riemannian optimization methods to the resulted smoothing problem. Then, we apply a homotopy technique during the iterations to obtain a good enough approximation by gradually decreasing the homotopy parameter. By combining Riemannian optimization methods, smoothing techniques,  and several additional assumptions, the proposed method can achieve better convergence results (computational complexity) than the existing subgradient methods. %In addition, when $f$ is a smooth function in the finite sum form $\sum_{i=1}^{n} f_i(X)$, we give a stochastic extension of our method and analysis its convergence result.

\subsection{Related works}
% \subsection{Nonsmooth manifold optimization}
% \paragraph{splitting-type method}

\subsubsection{ Lipschitz continuous function minimization} Riemannian gradient sampling methods are motivated by gradient sampling algorithms for nonconvex nonsmooth optimization in   Euclidean space. As introduced in \cite{hosseini2018line,hosseini2017riemannian}, given the current iterate $x_{k}$, a typical Riemannian gradient sampling algorithm first samples some points $\left\{x_{k}^{j}\right\}_{j=1}^{J}$ in the neighborhood of $x_{k}$ at which the objective function $f$ is differentiable, where the number of sampled points $J$ usually needs to be larger than the dimension of the manifold $\mathcal{M}$. Then,  it solves the following quadratic optimization problem for  obtaining  a descent direction: 
$$\boldsymbol{\xi}_{k}=-\underset{\boldsymbol{g} \in \operatorname{conv}(\mathcal{W})}{\operatorname{argmin}}\|\boldsymbol{g}\|^{2},$$
where $\operatorname{conv}(\mathcal{W})$ denotes the convex hull of $\mathcal{W}:=\left\{ g^1,\ldots, g^J \right\}$, $g^i \in T_{x^k}\mathcal{M}$ denotes a transported vector of $\operatorname{grad} f(x_{k}^{i}) $  by a vector transport, and $\operatorname{grad} f$ is the Riemannian gradient of $f$ on $\mathcal{M}$. The update can then be performed via classical retractions on $\mathcal{M}$ using the descent direction $\boldsymbol{\xi}_{k}$. This type of algorithm can potentially be utilized to solve a large class of Riemannian nonsmooth optimization problems. However, they are only known to converge asymptotically without any rate guarantee \cite{hosseini2018line,hosseini2017riemannian}. Furthermore, when addressing problem \eqref{prob:p1} in high dimensions using the Riemannian gradient sampling algorithm, a considerable number of Riemannian gradients must be sampled at each iteration. \blue{The} cost of solving the subproblem relies on the efficient computation of the vector transport operator.  The proximal point-type method iteratively computes the proximal mapping of the objective function over the Riemannian manifold \cite{de2016new,ferreira2002proximal}. The main issue with these methods is that each subproblem is as difficult as the original problem, which renders them not practical.

\subsubsection{Smoothing techniques for nonsmooth problem}

%Another related research is  the smoothing technique.  
In Euclidean space,    a variable smoothing algorithm  based on the Moreau envelope with a decreasing sequence of smoothing parameters was developed in \cite{bohm2021variable}, and it is proven to have a complexity of $\mathcal{O}(\epsilon^{-3})$ to achieve an $\epsilon$-approximate solution. We will extend this   work to  a class of nonsmooth problems on manifolds and develop an analog variable smoothing technique  on  compact Riemannian manifolds.  Such variable smoothing technique allows us to  resort to the analysis techniques for nonconvex minimization in the Euclidean space, and give some new convergence results for
the proposed  Riemannian smoothing gradient-type methods. Lin et al. \cite{lin2014smoothing} proposed a smoothing stochastic gradient method by incorporating the smoothing technique into the stochastic gradient descent algorithm, and Xu et al. \cite{xu2016homotopy} proposed a novel homotopy smoothing algorithm for solving a family of non-smooth problems which achieves the iteration complexity of $\mathcal{O}(\epsilon^{-1})$. For a class of nonsmooth convex problems with linear constraint, Wei et al. \cite{wei2018solving} proposed a primal-dual homotopy smoothing algorithm   achieving  lower complexity than $\mathcal{O}(\epsilon^{-1})$.
For  other works on the smoothing method, the readers are referred to  \cite{bot2019variable,metel2019simple,chambolle2011first,ouyang2012stochastic,tran2017adaptive}. 

The smoothing technique on   the Riemannian optimization problem is still scarcely explored. The authors in  \cite{liu2019simple} and  \cite{qu2019nonconvex} presented a smooth approximation for minimizing a nonsmooth function on   Riemannian manifolds, they did not utilize the Moreau envelope technique to obtain a smooth approximation, and the smoothing parameter is fixed. Cambier and Absil \cite{cambier2016robust} provided a homotopy smoothing strategy to solve robust low-rank matrix completion problems,  but they did not give the convergence analysis.  Zhang et al. \cite{zhang2021riemannian} proposed a Riemannian smoothing steepest descent method to
minimize a nonconvex and non-Lipschitz function on manifolds,  and gave an asymptotic convergence result.  More recently, Beck and ROSSET \cite{beck2023} proposed a dynamic smoothing gradient descent on manifold and obtained a convergence rate of $\mathcal{O}(1/k)$. It is worth noting that our study and this work were conducted concurrently (their article was published electronically on July 24, 2023), and their proposed algorithm does not include a randomized version.  
%To the best of our knowledge, this is the first work that employs the homotopy smoothing technique to nonsmooth problems on manifold and establishes non-asymptotic convergence rate analysis.

\subsubsection{Computational complexity on    Riemannian nonsmooth problem}
For the computational complexity of algorithms for solving nonsmooth and nonconvex minimization problems on Riemannian manifolds, Li et al. \cite{li2021weakly} presented a family of Riemannian subgradient-type methods and show that it has an iteration complexity of $\mathcal{O}(\epsilon^{-4})$ for driving a natural stationarity measure below $\epsilon$. However, if the function has more structure, it is possible to design algorithms that achieve better complexity. Indeed, if the
proximal operator of the nonsmooth function can be calculated analytically, i.e., $\mathcal{A} = \mathcal{I}$ in \eqref{prob:p1},  Chen et al. \cite{chen2020proximal}  proposed a retraction-based proximal gradient method in Stiefel manifold and achieves an iteration complexity of $\mathcal{O}(\epsilon^{-2})$ for obtaining an $\epsilon$-stationary solution. The Riemannian smoothing gradient method proposed in our paper achieves an iteration complexity of $\mathcal{O}(\epsilon^{-3})$, which interpolates between   $\mathcal{O}(\epsilon^{-2})$   and  $\mathcal{O}(\epsilon^{-4})$. It should be emphasized that each iteration in \cite{chen2020proximal} involves solving a subproblem that lacks an explicit solution, and they utilize the semismooth Newton method to solve it. Our algorithms require only one computation of the gradient for the smooth term $f$ and the computation of a proximity operator for the non-smooth term $h$ in each iteration. The authors in \cite{seguin2022continuation} 
  construct a suitable homotopy between the original manifold optimization problem and a problem that admits an easy solution, they develop and analyze a path-following numerical continuation
algorithm on manifolds for solving the resulting parameter-dependent problem. 
 Very recently, based on the smoothing technique, a Riemannian alternating direction method of multipliers (RADMM) is proposed in \cite{li2022riemannian}  with complexity result of $\mathcal{O}(\epsilon^{-4})$ for driving a Karush–Kuhn–Tucker (KKT) residual based stationarity. 

\subsection{Main contributions}
In this paper, we propose two Riemannian smoothing gradient-based methods for solving problem \eqref{prob:p1}: the Riemannian smoothing gradient method and the Riemannian smoothing stochastic gradient method. To analyze the convergence behavior of these methods, we introduce the concept of a generalized $\epsilon$-stationary point for a class of nonsmooth problems on manifolds. This concept serves as a weak version of the classical $\epsilon$-stationary point.
By employing the Moreau envelope technique, we derive a smoothing subproblem of the original problem \eqref{prob:p1} in Euclidean space. We establish the smoothness of this subproblem on the Riemannian compact submanifold. We demonstrate that the iterates generated by the aforementioned Riemannian smoothing gradient-type methods drive a generalized $\epsilon$-stationary point within computational complexities of $\mathcal{O}(\epsilon^{-3})$ and $\mathcal{O}(\epsilon^{-5})$, respectively. The corresponding proofs are presented in Theorems \ref{thm:full:epoch} and \ref{thm:sto:epoch}. Notably, these complexity guarantees align with the results in \cite{bohm2021variable}, which proposes a range of algorithms for solving composite weakly convex minimization problems in the Euclidean space. Numerical experiments in section \ref{sec:num} validate the superiority of the proposed algorithms.

\subsection{Notation}Throughout, Euclidean space, denoted by  $\mathbb{R}^n$, equipped with an inner product $\left<\cdot,\cdot\right>$ and inducing norm $\|x\| = \sqrt{\left<x,x\right>}$. Given a matrix $A$, we use $\|A\|_F$ to denote the Frobenius norm, $\|A\|_1:=\sum_{ij}\vert A_{ij}\vert$ to denote the $\ell_1$ norm. For a vector $x$, we use $\|x\|_2$ and $\|x\|_1$ to denote Euclidean norm and $\ell_1$ norm, respectively. The indicator function of a set $\mathcal{C}$, denoted by $\delta_{\mathcal{C}}$, is set to be zero on $\mathcal{C}$ and $+\infty$ otherwise. The distance from $x$ to $\mathcal{C}$ is denoted by $\mathrm{dist}(x,\mathcal{C}): = \min_{y\in\mathcal{C}}\|x-y\|$.
 For a differentiable function $f$ on $\mathcal{M}$, let grad $f(x)$  be its Riemannian gradient. If $f$ can be extended to the ambient Euclidean space, we denote its Euclidean gradient  by $\nabla f(x)$.

		\section{Preliminaries}\label{sec-preli}

In this section, we introduce some relevant concepts for Riemannian optimization, which can be regarded as some generalizations from Euclidean space to Riemannian manifolds. We refer the reader to \cite{AbsMahSep2008} for more details.
 We will also introduce subdifferential and the Moreau envelope technique.
\subsection{Riemannian optimization}
%A $d$-dimensional smooth manifold $\mathcal{M}$ is a Hausdorff and second-countable topological space, where each point has a neighborhood that is homomorphic to the $d$-dimensional Euclidean space. 
An $n$-dimensional smooth manifold $\mathcal{M}$ is an $n$-dimensional topological manifold equipped with a smooth structure, where each point has a neighborhood that is diffeomorphic to the $n$-dimensional Euclidean space.  For all $x\in\mathcal{M}$, there exists a chart $(U,\psi)$ such that $U$ is an open set and $\psi$ is a diffeomorphism between $U$ and an open set $\psi(U)$ in Euclidean space.   A tangent vector $\eta_x$ to $\mathcal{M}$ at $x$ is defined as tangents of parametrized curves $\gamma$ on $\mathcal{M}$ such that $\gamma(0) = x$ and
\[\nonumber
  \eta_x u : = \dot{\gamma}(0)u = \left.\frac{d(u(\gamma(t)))}{dt} \right\vert_{t=0} , \forall u\in  \oldwp_x\mathcal{M},
\]
where $\oldwp_x\mathcal{M}$ is the set of all real-valued functions $f$ defined in a neighborhood of $x$ in $ \mathcal{M}$. Then, the tangent space $T_x\mathcal{M}$ of a manifold $\mathcal{M}$ at $x$ is defined as the set of all tangent vectors at point $x$.  The manifold $\mathcal{M}$ is called a Riemannian manifold if it is equipped with an inner product on the tangent space $T_x\mathcal{M}$ at each $x\in\mathcal{M}$. In case that $\mathcal{M}$ is a Riemannian submanifold of Euclidean space $\mathcal{E}$, the inner product is defined as Euclidean inner product: $\left<\eta_x,\xi_x\right> = \mathrm{tr}(\eta_x^\top \xi_x)$. The Riemannian gradient $\grad  f(x) \in T_x\mathcal{M}$ is the unique tangent vector satisfying
$$  \left< \grad f(x), \xi \right> = df(x)[\xi], \forall \xi\in T_x\mathcal{M}. $$ If $\mathcal{M}$ is a compact Riemannian manifold embedded in Euclidean space, we have that $\grad f(x) = \mcP_{T_x\mathcal{M}}(\nabla f(x))$, where $\nabla f(x)$ is Euclidean gradient, $\mcP_{T_x \mathcal{M}}$ is the projection operator onto the tangent space $T_x \mathcal{M}$. 
%The Riemannian Hessian of $f$ at a point $x$ in $\mathcal{M}$ is the linear mapping $\Hess f(x)$ of $T_x\mathcal{M}$ into itself defined by $\Hess f(x)[\xi_x]  = \hat{\nabla}_{\xi_x}\grad f(x)$ for all $\xi_x $ in $T_x\mathcal{M}$, where $\hat{\nabla}$ is the Riemannian connection on $\mathcal{M}$. In particular, when $\mathcal{M}$ is an embedded manifold of Euclidean space, the Riemannian Hessian is defined as $ \Hess f(x)[z] = \mcP_{T_x\mathcal{M}}(\mbox{Dgrad}f(x)[z])$, where $\mbox{Dgrad}f(x)$ is the differential of $\grad f(x)$.
The retraction operator is one of the most important ingredients for manifold optimization, which turns an element of $T_x\mathcal{M}$ into a point in $\mathcal{M}$.

%\begin{definition}[Riemannian submanifold]
%Suppose $\mathcal{M}$ is a differentiable submanifold of $\mathcal{E}$. We call $\mathcal{M}$ to be a Riemannian submanifold of $\mathcal{E}$, if for any $x\in\mathcal{M}$ the tangent space $T_x\mathcal{M}$ is endowed with the Euclidean inner product; that is for any $\eta,\xi\in T_x\mathcal{M}$， if we let $T_x\mathcal{M}$ be embedded in $\mathcal{E}$ as a subspace, then inner product on $T_x\mathcal{M}$ is defined as $\left<\eta, \xi\right>_x: = \left<\eta,\xi\right>$, where the latter is the standard Euclidean inner product. Hence the norm $\|\cdot\|_x$ induced by $\left<\cdot,\cdot\right>_x$ is also the same as the standard $L_2$ norm .
%\end{definition}

\begin{definition}[Retraction]\label{def-retr}
  A retraction on a manifold $\mathcal{M}$ is a smooth mapping $\mcR:T\mathcal{M}\rightarrow \mathcal{M}$ with the following properties. Let $\mcR_x:T_x\mathcal{M} \rightarrow \mathcal{M}$ be the restriction of $\mcR$ at $x$:
\begin{itemize}
  \item $\mcR_x(0_x) = x$, where $0_x$ is the zero element of $T_x\mathcal{M}$,
  \item $d\mcR_x(0_x) = \mathrm{id}_{T_x\mathcal{M}}$,where $\mathrm{id}_{T_x\mathcal{M}}$ is the identity mapping on $T_x\mathcal{M}$.
\end{itemize}
\end{definition}

\subsection{Subdifferential and Moreau envelope}

% In this section, we present the smoothing technique, in which the nonsmooth term in objective function is replaced by its Moreau envelop, we analysis the properties of the Moreau envelop includes smoothness and convexity. Then, we will discuss the retraction-smoothness  of Moreau envelop over Riemannian manifold $\mathcal{M}$ when $\mathcal{M}$ is a Riemannian submanifold in Euclidean space.

Let $\varphi: \mathbb{R}^{n} \rightarrow(-\infty,+\infty]$ be a proper, lower semicontinuous, and extended real-valued function. The domain of $\varphi$ is defined as $\operatorname{dom}(\varphi)=\{{x} \in$ $\left.\mathbb{R}^{n}: \varphi({x})<+\infty\right\}$. A vector ${v} \in \mathbb{R}^{n}$ is said to be a Fr\'{e}chet subgradient of $\varphi$ at ${x} \in \operatorname{dom}(\varphi)$ if
\begin{equation}\label{eq:subdiff}
   \liminf _{ \substack{{y} \rightarrow {x}\\{y} \neq {x}} }
   \frac{\varphi({y})-\varphi({x})-\langle{v}, {y}-{x}\rangle}{\|{y}-{x}\|} \geq 0.
\end{equation}
The set of vectors ${v} \in \mathbb{R}^{p}$ satisfying \eqref{eq:subdiff} is called the Fr\'{e}chet subdifferential of $\varphi$ at ${x} \in \operatorname{dom}(\varphi)$ and denoted by $\widehat{\partial} \varphi({x})$. The limiting subdifferential, or simply the subdifferential, of $\varphi$ at ${x} \in$ $\operatorname{dom}(\varphi)$ is defined as
$$
\partial \varphi({x})=\left\{{v} \in \mathbb{R}^{n}: \exists {x}^{k} \rightarrow {x}, {v}^{k} \rightarrow {v} \text { with } \varphi({x}^{k}) \rightarrow \varphi({x}), {v}^{k} \in \widehat{\partial} \varphi ({x}^{k})\right\}.
$$
By convention, if ${x} \notin \operatorname{dom}(\varphi)$, then $\partial \varphi({x})=\emptyset.$ The domain of $\partial \varphi$ is defined as $\operatorname{dom}(\partial \varphi)=$ $\left\{{x} \in \mathbb{R}^{n}: \partial \varphi({x}) \neq \emptyset\right\}.$ For the indicator function $\delta_{\mathcal{S}}: \mathbb{R}^{n} \rightarrow\{0,+\infty\}$ associated with the non-empty closed set $\mathcal{S} \subseteq \mathbb{R}^{n}$, we have
$$
\widehat{\partial} \delta_{\mathcal{S}}({x})=\left\{{v} \in \mathbb{R}^{n}: \limsup _{{y} \rightarrow {x}, {y} \in \mathcal{S}} \frac{\langle{v}, {y}-{x}\rangle}{\|{y}-{x}\|} \leq 0\right\} \quad \text { and } \quad \partial \delta_{\mathcal{S}}({x})=\mathcal{N}_{\mathcal{S}}({x})  
$$
for any ${x} \in \mathcal{S}$, where $\mathcal{N}_{\mathcal{S}}({x})$ is the normal cone to $\mathcal{S}$ at ${x}$.
\begin{definition}[Weakly convex]
Function $g$ is said to be $\rho$-weakly convex if $g+\frac{\rho}{2}\|\cdot\|^2$ is convex for some $\rho\geq 0$. 
\end{definition}
\begin{definition}
 For a proper, $\rho$-weakly convex and lower semicontinuous function $h:\mathbb{R}^m \rightarrow \mathbb{R}$, the Moreau envelope of $h$ with the parameter $\mu\in(0,\rho^{-1})$ is given by
 \begin{equation*}
     h_{\mu}(y): = \inf_{z\in\mathbb{R}^m} \left\{h(z) + \frac{1}{2\mu }\|z- y\|^2  \right\}.
 \end{equation*}
\end{definition}

The provided lemma, stated as Lemma 4.1 in \cite{bohm2021variable}, establishes a relationship between the function values of two Moreau envelopes with distinct parameters.
\begin{lemma}\label{lemma-1}\cite[Lemma 4.1]{bohm2021variable} 
  Let $h$ be a proper, closed, and $\rho$-weakly convex function, and $h$ is also $\ell_h$-Lipschitz continuous.  Then 
  \begin{equation*}
      h_{\mu_2}(y) \leq h_{\mu_1}(y) + \frac{1}{2} \frac{\mu_1 - \mu_2}{\mu_2} \mu_1 \ell_h^2,
  \end{equation*}
 where $\mu_1$ and $\mu_2$  
 satisfy that  that $0< \mu_2 \leq \mu_1 <\rho^{-1}$.
\end{lemma}

The following result demonstrates that the Moreau envelope of a weakly convex function is not only continuously differentiable but also possesses a bounded gradient norm.
\begin{proposition}\label{propos-1}\cite[Lemma 3.3]{bohm2021variable}
Let $h$ be a proper, $\rho$-weakly convex, and $\ell_h$-Lipschitz continuous. Denote $\mu\in (0,\rho^{-1})$. Then,  Moreau envelope $h_{\mu}$ has Lipschitz continuous gradient over $\mbR^n$ with constant $\max\{\mu^{-1}, \frac{\rho}{1-\rho \mu}\}$, and the gradient is given by 
  \begin{equation}\label{moreau envelope gradient}
   \nabla h_{\mu}(x)  = \frac{1}{\mu}\left(x - \prox_{\mu h}(x) \right) \in \partial h\left(\prox_{\mu h}(x)\right),
  \end{equation}
where  $\prox_{\mu h}(x): = \arg\min_{y\in\mathbb{R}^n}
\{ h(y) + \frac{1}{2\mu}\|y-x\|^2 \}$ is the proximal operator. Moreover, it holds that
\begin{equation} \label{eq:moreau-gradient-bound}
\|\nabla h_{\mu}(x)\| \leq \ell_h, \; \|x - \prox_{\mu h}(x) \| \leq \mu \ell_h. 
\end{equation}
% In addition, $h_{\mu}$ also is continuously differentiable on $\mcM$ with Lipschitz constant $$M:=\max\{1/\mu, \frac{\rho}{1-\rho \mu}\}\alpha^2 +2 G\beta $$ in sense that
% $$
%     h_{\mu}(\mcR_{x}(\eta)) \leq h_{\mu}(x) + \left<\eta,\grad h_{\mu}(x)\right> + \frac{M}{2} \|\eta\|^2
% $$
%   for all $\eta\in T_{x}\mathcal{M}$. Its Riemannian gradient is given by $\grad h_{\mu}(x) = \mcP_{T_x\mcM}(\nabla h_{\mu}(x))$.

\end{proposition}

The Moreau envelope $h_{\mu}$ can be used for approximating nonsmooth function $h$, and  parameter $\mu$ is used to control the smoothness of $h_{\mu}$.  %Next, we will discuss the relationship between smoothness in Euclidean space and manifolds, which is crucial for the later convergence rate analysis of our algorithm. 
Finally, we provide the formal definition of retraction smoothness for a retraction operator $\mathcal{R}$. This concept plays a crucial role in the convergence analysis of the algorithms proposed in the subsequent section.
\begin{definition}\cite[Retraction smooth]{grocf}
  A function $f:\mathcal{M}\rightarrow\mathbb{R}$ is said to be  retraction smooth (short to retraction-smooth) with constant $\ell$ and a retraction $\mcR$,  if for   $\forall~x, y\in\mathcal{M}$  it holds that
  \begin{equation}\label{taylor expansion}
    f(y) \leq f(x) + \left<\grad f(x), \eta\right> + \frac{\ell}{2}\|\eta\|^2,
  \end{equation}
  where $\eta\in T_x\mathcal{M}$ and $y= \mcR_x(\eta)$.
\end{definition}

Throughout this paper, we make the following assumptions.
\begin{assumption}\label{assum}
%We have the following four assumptions.

\text{ }
\begin{enumerate}%[label={\theassumption.\textbf{\Alph*:}},
%  ref={\theassumption.\Alph*}]
 [label={\textbf{\Alph*:}},
  ref={\theassumption.\Alph*}]
  \item\label{assump-A}%
    The manifold $\mathcal{M}$ is a compact Riemannian submanifold embedded in  $\mathbb{E}$;
\item\label{assump-B}%
    The function $f$ is $\ell_{\nabla f}$-smooth but  not necessarily convex. % on $\mathbb{E}$,
  The function $h$ is    $\rho$-weakly convex  and $\ell_h$-Lipschitz continuous, but is not necessarily smooth; $\varphi$ is bounded from below.
  \item\label{assump-C}
  There exist constants $\alpha > 0 $ and $\beta>0$ such that, for all $x\in \mathcal{M}$ and  all $u \in T_{x}\mathcal{M}$, we have
   \begin{equation}\label{alpha}
  \left\{ \begin{aligned}
     \|\mathcal{R}_x(u) - x\| &\leq \alpha \|u\|, \\
     \|\mathcal{R}_x(u) - x - u\| &\leq \beta \|u\|^2.
   \end{aligned}\right.
   \end{equation}
\end{enumerate}
\end{assumption}

%
%\begin{lemma}
%  Let $\mathcal{E}$ be a Euclidean space (for example, $\mathcal{E} = \mathbb{R}^n$) and let $\mathcal{M}$ be a compact Riemannian submanifold of $\mathcal{E}$. Let $Retr$ be a retraction on $\mathcal{M}$. if $f:\mathcal{E}\rightarrow \mathbb{R}$ is convex in the convex hull of $\mathcal{M}$, then the pullbacks $f\circ \mbox{Retr}_x$ satisfy
%  \begin{equation}
%    f(\mbox{Retr}_{x^k}(\eta)) \geq F_k(x^k) + \left<\eta,\grad F_k(x^k)\right> - G\beta\|\eta\|^2
%  \end{equation}
%  for all $\eta\in T_{x^k}\mathcal{M}$.
%\end{lemma}
%
\begin{remark}
Assumption \ref{assump-C} is standard: The first inequality is common, and the second inequality follows from the fact that a retraction operator on a compact submanifold always satisfies a second-order boundedness property \cite{grocf}: $\mathcal{R}_x(u) = x + u + \mathcal{O}(\|u\|^2)$.
\end{remark}
\section{Stationary points}
In this section, we will discuss the stationary point of problem \eqref{prob:p1} and the smoothed problem induced by the Moreau envelope function. Recall the original problem, i.e., 
\begin{equation}\label{pro:p3}
    \min_x f(x) + h(\mathcal{A}x), ~~\mathrm{s.t. }~~ x\in\mathcal{M}.
\end{equation}
We say $x\in\mathcal{M}$ is  an $\epsilon$-stationary point of problem \eqref{pro:p3} if there exists $\zeta\in\partial h(\mathcal{A}x)$ such that $\|P_{T_x\mathcal{M}}(\nabla f(x) + \mathcal{A}^*\zeta)\| \leq \epsilon$. As in \cite{zhang2020complexity}, there is no finite time algorithm that can guarantee $\epsilon$-stationarity in the nonconvex nonsmooth setting.  Motivated by the notion of $(\delta,\epsilon)$-stationarity introduced in  \cite{zhang2020complexity}, we introduce a generalized $\epsilon$-stationarity for problem \eqref{pro:p3}. 
%We give the definition of the  generalized $\epsilon$-stationary point of problem \eqref{pro:p3} as follows.
\begin{definition}\label{defi-generalized}
 We say that, $y\in\mathbb{R}^n$ satisfies the generalized $\epsilon$-stationary condition of problem \eqref{pro:p3} if there exists $x\in\mathcal{M}$  such that
\begin{equation}\label{epsilon-kkt1}
    \left\{\begin{aligned}
    \mathrm{dist}\left(0,  P_{T_x\mathcal{M}}  (\nabla f(y) + \mathcal{A}^* \partial h(\mathcal{A}y))  \right) \leq \epsilon, \\
    \|x - y\| \leq \epsilon.
    \end{aligned}\right.
\end{equation}
\end{definition}

\begin{remark}
Note that $x$ is indeed a generalized $\epsilon$-stationary point if $x$ is an $\epsilon$-stationary point. If $\epsilon = 0$ in Definition \ref{defi-generalized}, the generalized $\epsilon$-stationarity will reduce  to $\epsilon$-stationarity (actually ‘stationarity’). Moreover, when $\mathcal{M} = \mathbb{R}^n$, our generalized $\epsilon$-stationarity coincides with the classical $\epsilon$-stationarity for minimizing a weakly convex function $\phi$, i.e., $\mathrm{dist}(0,\partial \phi(x)) \leq \epsilon$.

\end{remark}

Consider the associated smoothing problem:
\begin{equation}\label{pro:p3-smoothing}
\min_x  \quad  f(x)+h_{\mu}(\mathcal{A} (x)),\quad \text{s.t.}~~  x\in\mcM.
\end{equation}
We say that $x^*\in\mcM$ satisfies the $\epsilon$-stationary condition of problem \eqref{pro:p3-smoothing} if
\begin{equation}\label{equ:epsilon-station-smoothed-1}
  \|\mcP_{T_{x^*}\mcM}(\nabla f(x^*) + \mcA^* \nabla h_{\mu}(\mcA x^*))   \| \leq \epsilon.
\end{equation}
The relationship between the $\epsilon$-stationary point of problem \eqref{pro:p3} and \eqref{pro:p3-smoothing} is elaborated as follows. The proof is  followed by  \cite{bohm2021variable}.

\begin{lemma}\label{lemma-2} 
Suppose that Assumption \ref{assum} holds and $\mathcal{A}$ is surjective. If $x$ is an $\epsilon$-stationary point of problem \eqref{pro:p3-smoothing}, and  $$
0 < \mu \leq \min \left\{  \frac{1}{\ell_h\ell_{\nabla f}},\frac{2}{\ell_h} \right\} \sigma_{\min}(\mathcal{A})\epsilon,
 $$ where $\sigma_{\min}(\mathcal{A})$ is the smallest singular value of $\mathcal{A}$, and  $y$  satisfies  that  $\mcA y = \prox_{\mu h}(\mcA x)$, then  $y$ is a generalized $2\epsilon$-stationary point of problem \eqref{pro:p3}. 
\end{lemma}
 \begin{proof}
Let $y$ be the projection point to the set $\{ x^{\prime}: \mathcal{A}x^{\prime} =  \prox_{\mu h}(\mcA x)\}$, i.e., 
 \begin{equation}
     y = \min_{x^{\prime}\in \mathbb{R}^n} \{\| x - x^{\prime} \|^2 ~:~   \mathcal{A}x^{\prime}= \prox_{\mu h}(\mcA x)\},
 \end{equation}
which is given explicitly by  $y = x - \mathcal{A}^{\dagger}(\mathcal{A}x - \prox_{\mu h}(\mcA x) )$,  where $\mathcal{A}^{\dagger} = \mathcal{A}^*(\mcA\mcA^*)^{-1}$. Since $ f$ is $\ell_{\nabla f}$-Lipschitz continuous, $h$ is $\ell_h$-Lipschitz continuous and the norm of $\mcA^{\dagger}$ is bounded by $\sigma_{\min}^{-1}(\mcA)$, we have that $\| x-y \| \leq \sigma_{\min}^{-1}(\mcA)\mu \ell_h \leq 2\epsilon$.  Furthermore, follows  \eqref{equ:epsilon-station-smoothed-1} and Assumption \ref{assump-B} we have
     \begin{equation*}\label{eq:composite}
    \begin{aligned}
& \mathrm{dist}\left(0, \mcP_{T_{x}\mcM}\left(\nabla f(y) + \mcA^* \partial h(\mcA y)\right)   \right)\\
&= \mathrm{dist}\left(0,\mcP_{T_{x}\mcM}\left(\nabla f(y)-\nabla f(x) + \nabla f(x) + \mcA^* \partial h(\mcA y)\right)   \right) \\
& \leq \|\mcP_{T_{x}\mcM}\left(\nabla f(y)-\nabla f(x) \right)   \| + \mathrm{dist}\left(0, \mcP_{T_{x}\mcM}\left(\nabla f(x) + \mcA^* \partial h(\prox_{\mu h}(\mcA x))\right)   \right) \\ &\leq \ell_{\nabla f}\|x-y\| +   \|\mcP_{T_{x}\mcM}(\nabla f(x) + \mcA^* \nabla h_{\mu}(\mcA x))   \|  \\
  & \leq \ell_{\nabla f} \sigma_{\min}^{-1}(\mathcal{A}) \mu \ell_h    + \epsilon \leq 2\epsilon,
\end{aligned}
\end{equation*}
where the second inequality utilizes that the projection on tangent space is non-expansive and  $\nabla h_{\mu}(\mcA x) \in \partial h(\prox_{\mu h}(\mcA x)) $,  which is given  in Proposition \ref{propos-1}. 
 \end{proof}

	\section{Riemannian homotopy smoothing framework}

%We describe our variable smoothing approaches for the problem \eqref{prob:p1}. We define the smoothed approximation $F_k$ in $k$-iteration as follows: 
In this section, we focus on the problem \eqref{pro:p3}.  Let
\begin{equation}\label{eq:Fk}
F_k(x): = f(x) + h_{\mu_k}(\mcA x),    
\end{equation}
which is a smoothing approximation of $F(x)$. 
By Proposition \ref{propos-1} and chain rule, we obtain% its Euclidean gradient:
\begin{equation}\label{equ:gradient}
  \nabla F_k(x) = \nabla f(x) + \frac{1}{\mu_k} \mcA^*(\mcA x - \prox_{\mu_k h}(\mcA x)),
\end{equation}
and its Riemannian  gradient:
\begin{equation}\label{eq:riestochas2}
    \grad F_k(x) = \mcP_{T_x\mcM}(\nabla F_k(x )).
\end{equation}
The following lemma states that, if $\mathcal{M}$ is a compact submanifold of Euclidean space $\mathcal{E}$, then the smoothness of function on  a compact subset of $\mathcal{E}$ reduces to its retraction-smoothness over $\mathcal{M}$. This is an extension of Lemma 2.7 in \cite{grocf}.

\begin{lemma}\label{Euclidean vs manifold} 
  Let $\mathcal{E}$ be a Euclidean space $($for example, $\mathcal{E} = \mathbb{R}^n)$ and $\mathcal{M}$ be a compact Riemannian submanifold of $\mathcal{E}$. Suppose that Assumption \ref{assum} holds.  Then, for $F_k$  given in \eqref{eq:Fk}, there exists finite $G$ such that $\|\nabla F_k(x)\|\leq G$ for all $x\in\mathcal{M}$, and $F_k$ is retraction-smooth in the sense that
  \begin{equation}\label{equ:L2}
    F_k(\mcR_{x}(\eta)) \leq F_k(x) + \left<\eta,\grad F_k(x)\right> + \frac{\ell_k}{2}\|\eta\|^2
  \end{equation}
  for all $\eta\in T_{x}\mathcal{M}$, where $$\ell_k = \alpha^2 \ell_{\nabla f}  + \alpha^2 \|\mcA\|^2 \max\{\mu_k^{-1}, \frac{\rho}{1-\rho \mu_k}\}+2 G \beta.$$   If  $\rho \mu_k <1/2$, it reduces to  
\begin{equation}\label{equ:L}
    \ell_k = \alpha^2 \ell_{\nabla f}  + \alpha^2\|\mcA\|^2\mu_k^{-1} +2 G \beta.
\end{equation}
\end{lemma}
\begin{proof}
    By Proposition \ref{propos-1}, one shows that $\nabla F_k$  is Lipschitz continuous with constant $ \ell_{\nabla f}  + \|\mcA\|^2 \max\{\mu_k^{-1}, \frac{\rho}{1-\rho \mu_k}\}$.  Since $\nabla F_k$ is continuous on the compact manifold $\mathcal{M}$, there exists $G>0$ such that $\|\nabla F_k(x)\|\leq G$ for all $x\in\mathcal{M}$. The proof is completed by combining Assumption \ref{assump-C} with Lemma 2.7 in \cite{grocf}. 
\end{proof}

% By this lemma, we have
% $$
% \text{grad}F_k(x) = \mcP_{T_x\mcM}(\nabla F_k(x)).
% $$
% Let
% \begin{equation}\label{lips_constant}
% \ell_k = \alpha^2 \ell_{\nabla f}  + \|\mcA\|^2 \max\{1/\mu_k, \frac{\rho}{1-\rho \mu_k}\}\alpha^2 +2 G \beta
% \end{equation}
% which is the Lipschitz constant of the Riemannian gradient $\text{grad}F_k(x)$.

\subsection{Riemannian smoothing gradient method}Our first algorithm takes Riemannian gradient descent steps on the smoothed problem, which is
\begin{equation}\label{gradient_step}
  x^{k+1} = \mathcal{R}_{x^k}( - \gamma_k \grad   F_k(x^k)).
\end{equation}
The basic algorithm is described in Algorithm \ref{alg:full}. This algorithm is a gradient method and  advanced splitting techniques \cite{wang2019global,themelis2020douglas} can be adopted to minimize \eqref{eq:Fk}. Note that the Riemannian ADMM \cite{li2022riemannian} shows a favorable performance. It should be noted that our focus is to derive a sharper complexity analysis.  

\begin{algorithm}
 \caption{Riemannian smoothing gradient method for   \eqref{pro:p3}.}
\label{alg:full}
 \textbf{Input: }{ $x^1\in\mathcal{M}$}\\
  \While{the terminal criteria is not met}
  {
    Set $\mu_k = (2\rho)^{-1}k^{-1/3} $, let $\ell_k$   be as in \eqref{equ:L}, and  $\gamma_k = 1/\ell_k$;\\
    $x^{k+1} = \mathcal{R}_{x^k}( - \gamma_k\grad  F_k(x^k));$
 
 \textbf{end} }
\end{algorithm}

  With Assumption \ref{assum}, we are going to show the convergence results of the proposed Riemannian smoothing gradient method, Algorithm \ref{alg:full}. The main difference with its Euclidean analog  \cite[Algorithm 1]{bohm2021variable} lies in the retraction-smooth result given in Lemma \ref{lemma-1} and the analysis of the new stationarity given in Definition \ref{defi-generalized}. Denote 
  \begin{equation}\label{Fstar}
     F^*:=\liminf_{ k\rightarrow \infty}F_k(x^k).
  \end{equation}
  By Lemma \ref{lemma-1} and Assumption \ref{assump-B}, one can obtain that $F_k$ is bounded from below for all $k$.

%\begin{theorem}\label{theorem-1}
%Suppose in  problem \eqref{pro:p3} that,  $h$ is $\rho$-weakly convex and Lipschitz continuous with constant $\ell_h$, and $\nabla f$  is Lipschitz continuous with constant  $\ell_{\nabla f}$. The iteration sequence $\{x^j\}$ is generated by Algorithm \ref{alg:full}, and
%\begin{equation}\label{eq:xstar}
%x^*_j = x^j - \mathcal{A}^{\dagger}(\mathcal{A}x^j - \prox_{\mu_j h}(\mathcal{A}x^j)), \ \ j=1,\cdots,k.
%\end{equation}
%Then we have
% \begin{equation*}
%             \min_{1\leq j \leq k} \|\mcP_{T_{x^j}\mcM}\left(\nabla f(x^*_j) + \mcA^* \partial h(\mathcal{A} x^*_j)\right)   \| \leq     \frac{C}{k^{1/3}},
  %      \end{equation*}
 %      and \begin{equation}\label{xsta-x}
  %         \|x^j - x_j^*\| \leq \sigma_{\min}^{-1}(\mcA) \ell_h (2\rho)^{-1} j^{-1/3}, j= 1,\cdots, k. 
  %     \end{equation} 
  %     Where   $$C :=  2 \sqrt{\alpha^2 \ell_{\nabla f}  + 2G\beta + 2\rho \alpha^2 \|\mathcal{A}\|^2 } \sqrt{  F_1(x^1)  - F^* +(2\rho)^{-1} \ell_h^2 }.$$
%\end{theorem}

\begin{theorem}\label{theorem-1}
Suppose that Assumption \ref{assum}  holds and that the iteration sequence $\{x^k\}$ is generated by Algorithm \ref{alg:full}. Denote
\begin{equation}\label{eq:omega}
    \begin{aligned}
        \omega_1 & = \alpha^2 \ell_{\nabla f} +2G\beta + 2\rho   \alpha^2 \|\mathcal{A}\|^2, \\
        \omega_2 &= F_1(x^1) - F^*+ (2\rho)^{-1}\ell_h^2, \\
        \omega_3 &= (\ell_{\nabla f}\sigma_{\min}^{-1}(\mcA) (2\rho)^{-1.5} \ell_h)^2/(\alpha^2 \|\mathcal{A}\|^2),
    \end{aligned}
\end{equation}
where $G$ is defined in Lemma \ref{Euclidean vs manifold} and $F^*$ are defined \blue{as} \eqref{Fstar}.
Then we have for given $K\in \mathbb{N}$ that
\begin{equation}\label{eq:theorm2-min-1}
\begin{aligned}
    \min_{1\leq k \leq K} 
  \mathrm{dist}\left(0,\mcP_{T_{x^k}\mcM}(\nabla f(x^k) + \mcA^*  \partial h(\prox_{\mu_k h}(\mcA x^k)))   \right)  &\leq    \frac{2\sqrt{\omega_1\omega_2}}{K^{1/3}}, \\
   \|\mcA x^k - \prox_{\mu_k h}(\mcA x^k)\|  & \leq \frac{  (2\rho)^{-1}\ell_h}{ k^{1/3}}.
\end{aligned}
        \end{equation}
        
If $\mcA$ is also surjective, and $\hat{x}^k= x^k - \mathcal{A}^{\dagger}(\mcA x^k - \prox_{\mu_k h}(\mcA x^k)),$ then we have 
 \begin{equation}\label{eq:theorem2-min-2}
 \begin{aligned}
  \min_{1\leq k \leq K} \mathrm{dist}\left(0, \mcP_{T_{x^k}\mcM}(\nabla f(\hat{x}^k) + \mcA^* \partial h(\mcA \hat{x}^k))   \right) & \leq     \frac{2\sqrt{(\omega_1\omega_3 \ln(3K)   +2\omega_1 \omega_2) }}{K^{1/3}},\\
   \|x^k - \hat{x}^k\| & \leq
           \frac{\sigma_{\min}^{-1}(\mcA) (2\rho)^{-1} \ell_h}{ k^{1/3}}. 
 \end{aligned}
 \end{equation}
\end{theorem}

\begin{proof}
Since $\ell_k = 1/\gamma_k$ is the Lipschitz constant of $\grad F_k$, letting $\eta^k = -\gamma_k \grad F_k(x^k)$,  we have    
\begin{equation}\label{equ:in}
\begin{aligned}
    F_k(x^{k+1})  & \leq F_k(x^k) + \left<\grad F_k(x^k),  \eta^k  \right> + \frac{1}{2\gamma_k} \|\eta^k\|^2  \\
    & \leq F_k(x^k) - \frac{\gamma_k}{2} \|\grad F_k(x^k)\|^2,  ~~ \forall k\ge 1. 
\end{aligned}
    \end{equation}
  By Lemma \ref{lemma-1}, for all  $ x\in\mathcal{M}$, we have
    \begin{equation*}
      	\begin{aligned}
			F_{k+1}(x)  & \leq F_k(x) + \frac{1}{2} (\mu_k - \mu_{k+1}) \frac{\mu_k}{\mu_{k+1}} \| \nabla h_{\mu_k}(\mathcal{A}x)  \|^2 \\
			& \leq F_k(x) + (\mu_k - \mu_{k+1}) \ell_h^2, 
		\end{aligned}
    \end{equation*}
    where the second equality utilizes \eqref{eq:moreau-gradient-bound}.
   Set $x = x^{k+1}$ in the above inequality and add it to \eqref{equ:in}, we  obtain
    \begin{equation*}
        F_{k+1}(x^{k+1}) \leq F_k(x^k) - \frac{\gamma_k}{2} \| \grad F_k(x^k) \|^2 + (\mu_k - \mu_{k+1}) \ell_h^2.
    \end{equation*}
    By summing both sides of the above inequality  over $k = 1,2,\cdots,K $,  we get 
    \begin{equation}\label{equ:sum}
    \begin{aligned}
        \sum_{k=1}^{K} \frac{\gamma_k}{2} \| \grad F_k(x^k) \|^2  & \leq F_1(x^1) - F_{ K+1}(x^{ K+1}) +(\mu_1 - \mu_{ K+1}) \ell_h^2 \\
    & \leq F_1(x^1) - F^* + \mu_1 \ell_h^2. 
    \end{aligned}
        \end{equation}
By step 2 of Algorithm \ref{alg:full}, we have
        \begin{equation}\label{equ:gamak-right}
        \begin{aligned}
         \gamma_k=\frac{1}{\ell_k} &= \frac{\mu_k}{\mu_k (\alpha^2 \ell_{\nabla f} +2G\beta) + \alpha^2 \|\mathcal{A}\|^2}\\ & \geq   \frac{  (2\rho)^{-1}}{(2\rho)^{-1} (\alpha^2 \ell_{\nabla f} +2G\beta) + \alpha^2 \|\mathcal{A}\|^2} k^{-1/3}\\
         &= \omega_1^{-1} k^{-1/3}.
        \end{aligned}
        \end{equation}
      Combining the above inequality with \eqref{equ:sum}, we get
        \begin{equation}\label{equ:fina}
            \min_{1\leq k \leq K} \| \grad F_k(x^k)\|^2  \left(\frac{1}{2} \sum_{k=1}^{K} k^{-1/3} \right)\leq \omega_1\omega_2.
        \end{equation}
        Note that 
        \begin{equation*}
            \sum_{k=1}^K k^{-1/3} \geq \sum_{k=1}^K \int_k^{k+1} x^{-1 / 3} \mathrm{~d} x \geq \frac{1}{2} K^{2/3}.
        \end{equation*}
   It follows that
        \begin{equation*}
            \min_{1\leq k\leq K} \| \grad F_k(x^k)\|^2 \leq 4 \omega_1 \omega_2 \frac{1}{K^{2/3}}. \end{equation*}
 By combining this bound with Proposition \ref{propos-1}, and letting  $z^k : = \prox_{\mu_k h}(\mcA x^k)$, we obtain
   \begin{equation}\label{equ:fina2}
             \min_{1\leq k \leq K} \mathrm{dist}\left(0, \mcP_{T_{x^k}\mcM}(\nabla f(x^k) + \mcA^* \partial h(z^k))   \right) \leq  \min_{1\leq k\leq K} \| \grad F_k(x^k)\| \leq  \frac{2\sqrt{\omega_1\omega_2}}{K^{1/3}}.
        \end{equation}
It follows from Proposition \ref{propos-1} that %By Proposition \ref{propos-1},
%      $$\nabla 
% h_{\mu_k}(x^k)=\frac{1}{\mu_k}\left( x^k-\operatorname{prox}_{\mu_k h}( x^k)\right),$$ set $x^k  = \mathcal{A} x^k$ in the above inequality
$$\nabla 
h_{\mu_k}(\mathcal{A} x^k)=\frac{1}{\mu_k} \left(\mathcal{A} x^k-\operatorname{prox}_{\mu_k h}(\mathcal{A} x^k)\right),$$ Since $h$ is $\ell_h$-Lipschitz continuous, we have
  \begin{equation}\label{equ:fina22}
\|\mathcal{A} x^k-z^k\|= \| \mu_k \nabla 
h_{\mu_k}(\mathcal{A} x^k) \|  \leq   \mu_k \ell_h=\frac{(2 \rho)^{-1} \ell_h}{k^{1 / 3}}.
\end{equation}
%\blue{ The bounds \eqref{equ:fina2} and \eqref{equ:fina22} are not a perfect match with \eqref{equ:epsilon-station-smoothed-1}, since the subdifferentials of
%$f$ and $h \circ A$ are evaluated at different points. }
This implies that \eqref{eq:theorm2-min-1} holds.  If $\mcA$ is also surjective, the definition of $\hat{x}^k$ implies that
 \begin{equation}\label{eq:xstar2}
  x^k -\hat{x}^k= \mathcal{A}^{\dagger}(\mathcal{A}x^k - \prox_{\mu_k h}(\mathcal{A}x^k)).
\end{equation}
 Due to that the operator norm of $\mcA^{\dagger}$ is bounded by $\sigma_{\min}^{-1}(\mcA)$, we have
 \begin{equation}\label{equ:x-xstar}
     \begin{aligned}
     \|x^k -\hat{x}^k\| \leq  \sigma_{\min }^{-1}(\mathcal{A})\|\mathcal{A} x^k-z^k\|  \leq \frac{\sigma_{\min}^{-1}(\mcA) (2\rho)^{-1} \ell_h}{k^{1/3}}.
    \end{aligned}
 \end{equation} 
Furthermore, by Lemma  \ref{lemma-2}, we have
 that
     \begin{equation}\label{eq:composite2}
    \begin{aligned}
& \mathrm{dist} \left(0, \mcP_{T_{x^k}\mcM}(\nabla f(\hat{x}^k) + \mcA^* \partial h(\mcA \hat{x}^k))   \right)\\
%= &\|\mcP_{T_{x^k}\mcM}(\nabla f(\hat{x}^k)-\nabla f(x^k) + \nabla f(x^k) + \mcA^* \partial h(\mcA \hat{x}^k))   \| \\
 %\leq &\|\mcP_{T_{x^k}\mcM}(\nabla f(\hat{x}^k)-\nabla f(x^k) )   \| + \|\mcP_{T_{x^k}\mcM}(\nabla f(x^k) + \mcA^* \partial h(z^k))   \| \\
 \leq & \ell_{\nabla f}\|x^k-\hat{x}^k\| +  \mathrm{dist} \left(0,\mcP_{T_{x^k}\mcM}(\nabla f(x^k) + \mcA^* \partial h(z^k))   \right)\\
%& \leq L\|x-y\| + M (\kappa_f + L_g)  \|x-y\| + \epsilon \\
 \leq &   \frac{\ell_{\nabla f}\sigma_{\min}^{-1}(\mcA) (2\rho)^{-1} \ell_h}{k^{1/3}}+\| \grad F_k(x^k)\|.
\end{aligned}
\end{equation}
 Note that
       \begin{equation}\label{equ:gamak}
        \begin{aligned}
         \gamma_k = \frac{\mu_k}{\mu_k (\alpha^2 \ell_{\nabla f} +2G\beta) + \alpha^2 \|\mathcal{A}\|^2}  \leq   \frac{\mu_k}{\alpha^2 \|\mathcal{A}\|^2} = k^{-1/3}\frac{1}{2\rho \alpha^2 \|\mathcal{A}\|^2}.
        \end{aligned}
        \end{equation}
 This implies that
  \begin{equation}
  \begin{aligned}
     & \sum_{k=1}^K  \frac{\gamma_k}{4 } \mathrm{dist}^2\left(0,\mcP_{T_{x^k}\mcM}\left(\nabla f(\hat{x}^k) + \mcA^* \partial h(\mcA \hat{x}^k)\right)  \right) \\
      & \leq  \sum_{k=1}^K \frac{\gamma_k}{2}\frac{(\ell_{\nabla f}\sigma_{\min}^{-1}(\mcA) (2\rho)^{-1} \ell_h)^2}{k^{2/3}} +  \sum_{k=1}^K \frac{\gamma_k}{2} \| \grad F_k(x^k)\|^2\\
      & \leq \frac{\omega_3}{2} \sum_{k=1}^K \frac{1}{k}   +  \omega_2 \leq \frac{\omega_3}{2} \ln(3K)   + \omega_2,
  \end{aligned}
  \end{equation}
  where the first inequality is based on the fact that $(a+b)^2 \leq 2(a^2 +b^2)$ for any $a,b>0$, the second inequality utilize \eqref{equ:sum}, \eqref{equ:gamak} and \eqref{equ:x-xstar}. Similar to the above analysis, one can obtain that
  \begin{equation}
            \min_{1\leq k \leq K} \mathrm{dist}^2\left(0,\mcP_{T_{x^k}\mcM}\left(\nabla f(\hat{x}^k) + \mcA^* \partial h(\mcA \hat{x}^k)\right)  \right) \leq  8\omega_1(\frac{\omega_3}{2} \ln(3K)   + \omega_2) \frac{1}{K^{2/3}}.
        \end{equation}
    Combining with \eqref{equ:x-xstar} implies that \eqref{eq:theorem2-min-2} holds. The proof is completed. 
\end{proof}

% \begin{corollary}
% Show the complexity result about $\epsilon$.    
% \end{corollary}

% This theorem has a discrepancy between the two boundaries  \eqref{eq:theorm2-min-1} and \eqref{eq:theorem2-min-2}. The first bound  \eqref{eq:theorm2-min-1}  states that we will come across an iteration $k$ where the first-order optimality condition is satisfied to within a tolerance of $\epsilon$ during the first $K=\mathcal{O}(\epsilon^{-3})$ iterations. The second bound on $\|x^k-\hat{x}^k\|$, however, may not be particularly small because it might have been satisfied at an early iteration (i.e., $k \ll \epsilon^{-3}$). The algorithm used to correct this flaw is described in the following section.
 The theorem reveals a discrepancy between the two boundaries, specifically \eqref{eq:theorm2-min-1} and \eqref{eq:theorem2-min-2}. While the first bound, \eqref{eq:theorm2-min-1}, guarantees that an iteration $k$ will be reached where the first-order optimality condition is satisfied within a tolerance of $\epsilon$ within the first $k=\mathcal{O} (\epsilon^{-3})$ iterations, the second bound on $\|x^k-\hat{x}^k\|$ does not necessarily ensure a sufficiently small value, as it may have been satisfied in an early iteration ($k \ll \epsilon^{-3}$). To address this discrepancy, we introduce a dedicated algorithm in the subsequent subsection.

\subsection{Riemannian smoothing gradient method with epochs}

We describe a variant of Algorithm \ref{alg:full} in which the steps are organized into a series of epochs, each of which is twice as long as the one before. This variant is described in Algorithm \ref{alg:full-epoch}.   
We will show that,  there is some iteration $k=O(\epsilon^{-3})$ such that both $\|\mathcal{A} x^{k}-\operatorname{prox}_{\mu_{k} h}(\mathcal{A} x^{k})\|$ and $
\mathrm{dist}\left(0,\mcP_{T_{x^k}\mcM}(\nabla f(y^k) + \mcA^* \partial h(\mcA y^k))   \right) 
$
are smaller than the given tolerance $\epsilon$.
\begin{algorithm}
\caption{Riemannian smoothing gradient method with epochs for  \eqref{pro:p3}.} \label{alg:full-epoch}
  \textbf{Input:}{$x^1\in\mathcal{M}$ and tolerance  $\epsilon>0$; }\\
  \For{$l = 0,1,2,\cdots$}
  {
   set $S_l$ $\leftarrow$ $\infty$, and $k_l$ $\leftarrow$ $2^l$;\\
  \For{$k = 2^l,2^l+1,\cdots,2^{l+1}-1 $}{
 set $\mu_k = (2\rho)^{-1}k^{-1/3} $, let $\ell_k$ be as in \eqref{equ:L}, set $\gamma_k = 1/\ell_k$;\\
 $x^{k+1} = \mathcal{R}_{x^k}( - \gamma_k\grad  F_k(x^k))$;\\
\If{$\|\operatorname{grad} F_{k+1}(x^{k+1})\| \leq S_{l}$}{
 $S_l \leftarrow \|\operatorname{grad} F_{k+1}(x^{k+1})\| $,   $k_l \leftarrow k+1$;\\
\If{$S_{l} \leq \epsilon$ and $\|\mathcal{A} x^{k+1}-\operatorname{prox}_{\mu_{k+1} h}(\mathcal{A} x^{k+1})\|  \leq \epsilon $}{
break;\\
\textbf{end}}
\textbf{end}}
\textbf{end}}
\textbf{end}}
\end{algorithm}

\begin{proposition}
    \label{pro:epoch}
Suppose that Assumption \ref{assum} holds.  Consider using Algorithm \ref{alg:full-epoch} to  problem \eqref{pro:p3},  and $F^{*}$  is finite. Let $\omega_1,\omega_2,\omega_3$ be constants defined by Theorem \ref{theorem-1}. Then, for a given tolerance $\epsilon>0$ and  some $k=O(\epsilon^{-3})$, the iterate $x^{k}$  satisfies that
\begin{equation*}
   	\mathrm{dist}\left(0,\mathcal{P}_{T_{x^{k}} \mathcal{M}}(\nabla f(x^{k})+\mathcal{A}^{*} \partial h(\operatorname{prox}_{\mu_{k} h}(\mathcal{A} x^{k})))\right) \leq \epsilon \text { and }\|\mathcal{A} x^{k}-\operatorname{prox}_{\mu_{k} h}(\mathcal{A} x^{k})\| \leq \epsilon.
\end{equation*}
\end{proposition}

\begin{proof}
    As in \eqref{equ:sum}, by using monotonicity of $\left\{\grad F_{k}(x_{k})\right\}$ and discarding nonnegative terms, we have that
	\begin{equation}\label{equ:theo3:sum-grad}
	  	\sum_{k=2^{l}}^{2^{l+1}-1} \frac{\gamma_{k}}{2}\|\operatorname{grad} F_{k}(x^{k})\|^{2} \leq \omega_2.
	\end{equation}
	With the same arguments as in the earlier proof, we obtain
	$$
	\begin{aligned}
		\sum_{k=2^{l}}^{2^{l+1}-1} k^{-1 / 3} & \geq \sum_{k=2^{l}}^{2^{l+1}-1} \int_{k}^{k+1} x^{-1 / 3} \mathrm{~d} x=\int_{2^{l}}^{2^{l+1}} x^{-1 / 3} \mathrm{~d} x \\
		&=\frac{3}{2}\left((2^{l+1})^{2 / 3}-(2^{l})^{2 / 3}\right)=\frac{3}{2}\left(2^{2 / 3}-1\right)(2^{l})^{2 / 3} \geq \frac{1}{2}(2^{l})^{2 / 3}.
	\end{aligned}
	$$
	Substituting this inequality  and \eqref{equ:gamak-right} into \eqref{equ:theo3:sum-grad} yields
	$$
	\min _{2^{l} \leq k \leq 2^{l+1}-1}\|\operatorname{grad} F_{k}(x^{k})\| \leq \frac{2\sqrt{\omega_1\omega_2}}{(2^{l})^{1 / 3}}.
	$$
 Let $z^{k}:=$ $\operatorname{prox}_{\mu_{k} h}(\mathcal{A} x^{k})$, we have as in \eqref{equ:fina2}  that
	\begin{equation}\label{eq:epoch-1}
	    \min _{2^{l} \leq k \leq 2^{l+1}-1} \|\mathcal{P}_{T_{x^{k}} \mathcal{M}}(\nabla f(x^{k})+\mathcal{A}^{*} \partial h(z^{k}))\| \leq \frac{2\sqrt{\omega_1\omega_2}}{(2^{l})^{1 / 3}}.
	\end{equation}
Furthermore, for all  $2^{l} \leq k \leq 2^{l+1}-1$ we have 
	\begin{equation}\label{eq:epoch-2}
	    	\|\mathcal{A} x^{k}-z^{k}\| \leq  \mu_{k} \ell_{h}\leq \frac{(2 \rho)^{-1} \ell_{h}}{k^{1 / 3}} \leq \frac{(2 \rho)^{-1} \ell_{h}}{(2^{l})^{1 / 3}}.
	\end{equation}
By \eqref{eq:epoch-1} and \eqref{eq:epoch-2} we deduce that Algorithm \ref{alg:full-epoch} must terminate before the end of epoch $l$, i.e., before $2^{l+1}$ iterations, where $l$ is the smallest nonnegative integer such that
	$$
	2^{l} \geq \max \left\{{8\sqrt{\omega_1^{3}\omega_2^{3}}},(2 \rho)^{-3} \ell_{h}^{3}\right\} \epsilon^{-3}.
	$$
	Thus, termination occurs  at most $ 2\max \left\{8\sqrt{\omega_1^{3}\omega_2^{3}},(2 \rho)^{-3} \ell_{h}^{3}\right\} \epsilon^{-3}$ iterations.
\end{proof} 
	For the case of that $\mcA$ is surjective, we have the following stronger result.

\begin{theorem}\label{thm:full:epoch}
    Suppose that Assumption \ref{assum} holds.  Let $\omega_1,\omega_2,\omega_3$ be constants defined by Theorem \ref{theorem-1}. Assume that $\mcA$ is also surjective, and that the condition $\|\mathcal{A} x^{k+1}-\operatorname{prox}_{\mu_{k+1} h}(\mathcal{A} x^{k+1})\| \leq \epsilon$ in Algorithm \ref{alg:full-epoch} is replaced by $\|x^{k+1}-\hat{x}^{k+1}\| \leq \epsilon$, where $\hat{x}^{k+1} = x^{k+1} - \mathcal{A}^{\dagger}(\mathcal{A}x^{k+1} - \prox_{\mu_{k+1} h}(\mathcal{A}x^{k+1}))$. Then, for some $k^{\prime}=O(\epsilon^{-3})$, we have that
	\begin{equation*}
	    \mathrm{dist}\left(0,\mathcal{P}_{T_{x^{k}} \mathcal{M}}\left(\nabla f(x_{k^{\prime}}^{*})+\mathcal{A}^{*} \partial h(\mathcal{A} \hat{x}^{k^{\prime}})\right)\right) \leq \epsilon,~~\left\|x^{k^{\prime}}-\hat{x}^{k^{\prime}}\right\| \leq \epsilon.
	\end{equation*}

\end{theorem} 
	
\begin{proof}
	If $\mcA$ is surjective, the operator norm of $\mcA^{\dagger}$ is bounded by $\sigma_{\min}^{-1}(\mcA)$. Therefore, for all  $2^{l} \leq k \leq 2^{l+1}-1$, the definition of $\hat{x}^k$ implies that
 \begin{equation}\label{equ:x-xstar1}
     \begin{aligned}
     \|x^k -\hat{x}^k\| \leq  \sigma_{\min }^{-1}(\mathcal{A})\|\mathcal{A} x^k-z^k\|  \leq \frac{\sigma_{\min}^{-1}(\mcA) (2\rho)^{-1} \ell_h}{(2^l)^{1/3}}.
    \end{aligned}
 \end{equation} 
	Note that 
	\begin{equation}\label{equ:k-inequality}
    \begin{aligned}
   \sum_{k=2^l}^{2^{l+1}-1} k^{-1} & \leq  \sum_{k=2^l}^{2^{l+1}-1} \int_{k-1}^{k}  x^{-1} dx  = \int_{2^l-1}^{2^{l+1}-1} x^{-1} dx \leq \ln(2^{l+1}) - \ln(2^{l-1}) =  \ln(4).
\end{aligned}
\end{equation}
	Combining \eqref{eq:composite2},\eqref{equ:theo3:sum-grad} and \eqref{equ:k-inequality}, we have 
  \begin{equation}
  \begin{aligned}
     & \sum_{k=2^l}^{2^{l+1}-1}  \frac{\gamma_k}{4 } \mathrm{dist}^2\left(0,\mcP_{T_{x^k}\mcM}\left(\nabla f(\hat{x}^k) + \mcA^* \partial h(\mcA \hat{x}^k)\right)   \right) \\
      & \leq  \sum_{k=2^l}^{2^{l+1}-1} \frac{\gamma_k}{2}\frac{(\ell_{\nabla f}\sigma_{\min}^{-1}(\mcA) (2\rho)^{-1} \ell_h)^2}{k^{2/3}} +  \sum_{k=2^l}^{2^{l+1}-1} \frac{\gamma_k}{2} \| \grad F_k(x^k)\|^2\\
      & \leq \frac{\omega_3}{2} \sum_{k=2^l}^{2^{l+1}-1} \frac{1}{k}   +  \omega_2 \leq \frac{\omega_3}{2} \ln(4)   + \omega_2\leq \omega_2 + \omega_3.\\
  \end{aligned}
  \end{equation}
   Similar to the above analysis, one can obtain that
  \begin{equation}
            \min_{2^l\leq k \leq 2^{l+1}-1} \mathrm{dist}\left(0,\mcP_{T_{x^k}\mcM}\left(\nabla f(\hat{x}^k) + \mcA^* \partial h(\mcA \hat{x}^k)\right)   \right) \leq  \frac{2\sqrt{2}\sqrt{\omega_1(\omega_2+\omega_3)}}{(2^l)^{1/3}}.
        \end{equation}
 With the considerations made in the previous proof, we can choose $l$ to be the smallest positive integer such that
	$$
	2^{l+1} \geq 2 \max \left\{8\sqrt{\omega_1^{3}(\omega_2+\omega_3)^{3}}, \sigma_{\min }^{-3}(\mathcal{A}) \ell_{h}^{3}(2 \rho)^{-3}\right\} \epsilon^{-3}.
	$$
	The claim then holds for some $k^{\prime} \leq 2^{l+1}$.
\end{proof}

	%Although Algorithm \ref{alg:full-epoch} seems more complicated than Algorithm \ref{alg:ssn}, the steps are the same. The only difference is that for the second algorithm, we do not search for the iterate that minimizes criticality across all iterations but only across at most the last $k / 2$ iterations, where $k$ is the total number of iterations.

\subsection{Riemannian smoothing stochastic gradient method}
Now we turn to analyze the Riemannian smoothing stochastic gradient method. In particular, we consider the following stochastic optimization problem on  manifolds:
\begin{equation}\label{eq:rsg-22}
    \min_{x\in\mcM} F(x):= \mbE_{\xi\in\mcD}[f(x,\xi)] + h(\mcA x).
\end{equation}
Here, we assume that the function $h$ is $\rho$-weakly convex and $\ell_h$-Lipschitz continuous, $f(x): = \mbE_{\xi\in\mcD}[f(x,\xi)]$ is $\ell_{\nabla f}$-Lipschitz continuous. In the $k$-th iteration, we define the smoothed approximation $F_k(x,\xi): = f(x,\xi) + h_{\mu_k}(\mcA x)$ and 
\begin{equation}
F_k(x ): = \mbE_{\xi\in\mcD}[ F_k(x,\xi)] =   \mbE_{\xi\in\mcD}[f(x,\xi)]  + h_{\mu_k}(\mcA x).
\end{equation}
 Proposition \ref{propos-1} implies that $F_k$ is continuous differentiable with gradient
$$
  \begin{aligned}
    \nabla F_k(x) &= \mbE_{\xi\in\mcD}[\nabla f(x,\xi)] + \frac{1}{\mu_k} \mcA^*(\mcA x - \prox_{\mu_k h}(\mcA x)).
     \end{aligned}
$$
The Riemannian  gradient is as follows:
\begin{equation}\label{eq:riestochas}
    \grad F_k(x) = \mcP_{T_x\mcM}(\nabla F_k(x )).
\end{equation}
By Lemma \ref{Euclidean vs manifold},  when $\rho \mu_k \leq 1/2$, $F_k$ satisfies \eqref{equ:L2} with constant $\ell_k$ given by
\begin{equation}\label{equ:L22}
    \ell_k = \alpha^2 \ell_{\nabla f}  + \alpha^2\|\mcA\|^2\mu_k^{-1} +2 G \beta.
\end{equation}
% Let $G(x,\xi)$ denote the stochastic gradient of $f(x,\xi)$ respect to $x$,

Let $\nabla F_k(x,\xi)$ and $\grad F_k(x,\xi)$  denote the Euclidean stochastic gradient and Riemannian stochastic gradient of $F_k(x)$, respectively, i.e.,
$$
\begin{aligned}
    \nabla F_k(x,\xi): &= \nabla f(x,\xi) + \frac{1}{\mu_k} \mcA^*(\mcA x - \prox_{\mu_k h}(\mcA x)), \\
    \grad F_k(x,\xi) : &= \mathcal{P}_{T_x\mathcal{M}}\left(  \nabla F_k(x,\xi) \right). 
    \end{aligned}
$$
The following assumptions are made for the stochastic gradient $\nabla f(x,\xi)$:
\begin{assumption}\label{assum:stochatis}
We make the following assumptions:
\begin{itemize}
    \item %{\color{red} The algorithm can generate independent and identically distributed samples according to the distribution $D$.} 
    The samples used in the proposed algorithms are independent and identically distributed with distribution $\mathcal{D}$. 
    \item For any $x$, the algorithm generates a sample $\xi\sim \mathcal{D}$ and returns a stochastic gradient $\nabla f(x,\xi)$, there exists a parameter $\sigma>0$ such that
\begin{align}
  &  \mathbb{E}_{\xi}\left[ \nabla f(x,\xi) \right] = \nabla f(x), \label{assm:mean} \\ 
 &    \mathbb{E}_{\xi}\left[ \|\nabla f(x,\xi) - \nabla f(x) \|^2\right] \leq \sigma^2. \label{assm:vari}
\end{align}
\end{itemize}
\end{assumption}
Under this assumption,  for any $k>1$ we can conclude that
\begin{equation}\label{eq:gradient}
\begin{aligned}
     \mathbb{E}_{\xi}\left[\grad F_k(x, \xi) \right] & =\mcP_{T_x\mcM}\left(  \nabla f(x) + \frac{1}{\mu_k} \mcA^*(\mcA x - \prox_{\mu_k h}(\mcA x)) \right) = \grad F_k(x),
     \end{aligned}
\end{equation}
\begin{equation}\label{eq:variance}
\begin{aligned}
 \mathbb{E}_{\xi}\left[\|\grad F_k(x, \xi )-\grad F_k(x)\|^{2}\right] & = \mathbb{E}_{\xi}\left[\|
     \mcP_{T_x\mcM}\left( \nabla F_k(x, \xi )-\nabla F_k(x) \right)\|^{2}\right] \\ 
&  \leq \mathbb{E}_{\xi}\left[\| \nabla F_k(x, \xi )-\nabla F_k(x) \|^{2}\right] \\ 
%& =\mathbb{E}_{\xi}\left[\left\| ( G(x, \xi) + \frac{1}{\mu_k} \mcA^*(\mcA x - \prox_{\mu_k h}(\mcA x)))-( \nabla f(x) + \frac{1}{\mu_k} \mcA^*(\mcA x - \prox_{\mu_k h}(\mcA x))) \right\|^{2}\right] \\ 
&     =  \mathbb{E}_{\xi}\left[\|\nabla f(x, \xi)-\nabla f(x)\|^{2}\right] \leq \sigma^{2}.
\end{aligned}
\end{equation}
At iteration $k$, the Riemannian smoothing stochastic gradient algorithm generates a sample $\xi_k\sim \mathcal{D}$ that is independent of $\{\xi_0,\cdots, \xi_{k-1}\}$, and  returns a Riemannian stochastic gradient $\grad F_k(x^k,\xi_k)$.  Then it generates the next iterate $x^{k+1}$ via
\begin{equation*}
    x^{k+1} = \mcR_{x^k}(-\gamma_k \grad F_k(x^k,\xi_k)).
\end{equation*}
The detail of the Riemannian smoothing stochastic gradient algorithm is presented in Algorithm \ref{alg:stochastic}.

% for some parameter $\sigma \geq 0$. Observe that, by\eqref{eq:gradient} $, G(x^{k}, \xi_{k})$ is an unbiased estimator of $\nabla f(x^{k})$ and, by \eqref{eq:variance}, the variance of the random variable $\|G(x^{k}, \xi_{k})-\nabla f(x^{k})\|$ is bounded. It is worth noting that in the standard setting for stochastic programming, the random vectors $\xi_{k}, k=1,2, \ldots$, are independent of each other (and also of $x^{k}$ ). Our assumption here is slightly weaker since we do not need to assume $\xi_{k}, k=1,2, \ldots$, to be independent.

\begin{algorithm}
\caption{Riemannian smoothing stochastic gradient method for  \eqref{eq:rsg-22}.} \label{alg:stochastic}
 \textbf{Input:}{ Initial point $x^{1}$, iteration limit $K$, stepsizes $\left\{\gamma_{k}\right\}_{k \geq 1}$. }\\
  \For{$k=1, \ldots, K$}
  {
   Sample $\xi_k \sim  \mathcal{D}$, let $\mu_k = (2\rho)^{-1}k^{-1/5} $ and   $\ell_k$ be as in  \eqref{equ:L22}.\\
   Compute the Riemannian stochastic gradient via \eqref{eq:riestochas}.\\
   Update $x^{k+1}$ via
\begin{equation}\label{eq:rsg}
x^{k+1} = \mathcal{R}_{x^k}( - \gamma_k\grad  F_k(x^{k},\xi_k )).
\end{equation}
\textbf{end}}
Sample $R\in\{1,\cdots,K\}$ according to $\mathrm{Prob}(R = k) = \frac{2 \gamma_{k}-\ell_{k} \gamma_{k}^{2}}{\sum_{k=1}^{N}\left(2 \gamma_{k}-\ell_{k} \gamma_{k}^{2}\right)}$. \\
\textbf{Output: }{ $x^{R}$.}
\end{algorithm}

The following result gives some convergence properties of Algorithm \ref{alg:stochastic}.

\begin{theorem}\label{thm:ssumptionA1}
Suppose that Assumption \ref{assum} and \ref{assum:stochatis} hold,, and the stepsizes $\left\{\gamma_{k}\right\}$ are chosen such that $\gamma_{k}<2 / \ell_{k}$. Let $x^R$ be the point returned by Algorithm \ref{alg:stochastic}. Then 
  we have
\begin{equation}\label{theorem2}
   \mathbb{E}\left[\|\grad
   F_R(x^{R})\|^{2}\right] \leq \frac{2 \omega_2 +\sigma^{2} \sum_{k=1}^{K} \ell_k\gamma_{k}^2}{\sum_{k=1}^{K}\left(2 \gamma_{k}-\ell_{k} \gamma_{k}^{2}\right)},
\end{equation}
where $\omega_2$ is defined in Theorem \ref{theorem-1}, and the expectation is taken over all random choices.
\end{theorem}

\begin{proof}

Let $\eta^k=-\gamma_k \grad F_k(x^{k}, \xi_{k} )$ and  $\delta_{k}\equiv\grad F_k(x^{k}, \xi_{k} )-\grad F_k(x^{k}), \forall k \geq 1$. Using the assumption that $F \in \mathcal{C}_{L}^{1,1}\left(\mathbb{R}^{n}\right)$, retraction-smooth property \eqref{taylor expansion} and iterate formulation  \eqref{eq:rsg}, we have
\begin{equation}\label{theorem466}
\begin{aligned}
\mathbb{E}_{\xi_k} & \left[ F_{k}(x^{k+1}) \right] \leq  F_k(x^{k})+ \mathbb{E}_{\xi_k} \left[\left\langle  \grad  F_k(x^{k}), \eta ^k \right\rangle \right]+\frac{\ell_{k}}{2}  \mathbb{E}_{\xi_k} \left[ \|  \eta^k \|^{2} \right] \\
%=& F_k(x^{k})+\mathbb{E}_{\xi_k}\left[\left\langle  \grad  F_k(x^{k}), -\gamma_{k}  \grad  F_k(x^{k}, \xi_{k} )\right\rangle\right]+\frac{\ell_{k}}{2} \mathbb{E}_{\xi_k} \left[ \| -\gamma_{k} \grad  F_k(x^{k}, \xi_{k} )\|^{2} \right] \\
%=& F_k(x^{k})-\gamma_{k}\mathbb{E}_{\xi_k}\left[ \left\langle  \grad  F_k(x^{k}),   \grad  F_k(x^{k}, \xi_{k} )\right\rangle \right]+\frac{\ell_{k}}{2} \gamma_{k}^{2}\mathbb{E}_{\xi_k} \left[ \|  \grad  F_k(x^{k}, \xi_{k} )\|^{2} \right] \\
=& F_k(x^{k})-\gamma_{k}\mathbb{E}_{\xi_k}\left[ \left\langle  \grad  F_k(x^{k}),   \grad  F_k(x^{k} )+ \delta_k  \right\rangle \right] +\frac{\ell_{k}}{2} \gamma_{k}^{2}\mathbb{E}_{\xi_k} \left[ \|   \grad  F_k(x^{k} )+ \delta_k  \|^{2} \right] \\
=& F_k(x^{k})-\gamma_{k} \left(
\|  \grad  F_k(x^{k})\|^{2}+\mathbb{E}_{\xi_k} \left[\left\langle  \grad  F_k(x^{k}), \delta_{k}\right\rangle \right] \right) \\
&+ \frac{\ell_{k}}{2} \gamma_{k}^{2}\left[\|  \grad  F_k(x^{k})\|^{2}+2 \mathbb{E}_{\xi_k} \left[\left\langle  \grad  F_k(x^{k}), \delta_{k}\right\rangle \right]+\mathbb{E}_{\xi_k} \left[\|\delta_{k}\|^{2}\right]\right] \\
=& F_k(x^{k})-\left(\gamma_{k}-\frac{\ell_{k}}{2} \gamma_{k}^{2}\right)\|  \grad  F_k(x^{k})\|^{2}\\
&-\left(\gamma_{k}-\ell_{k} \gamma_{k}^{2}\right)\mathbb{E}_{\xi_k} \left[\left\langle  \grad  F_k(x^{k}), \delta_{k}\right\rangle \right]+\frac{\ell_{k}}{2} \gamma_{k}^{2}\mathbb{E}_{\xi_k} \left[\|\delta_{k}\|^{2}\right],
\end{aligned}
\end{equation}
where the first inequality utilizes \eqref{equ:L2} and \eqref{eq:rsg}. 
It follows from \eqref{eq:gradient} that
\begin{equation}\label{theorem6}
\begin{aligned}
& \mathbb{E}_{\xi_{k}}\left[\left\langle \grad F_k(x^{k}), \delta_{k}\right\rangle \right] = 0.
%&  = \mathbb{E}_{\xi_{k}}\left[\left\langle \grad F_k(x^{k}), \grad F_k(x^{k}, \xi_{k} )-\grad F_k(x^{k}) \right\rangle \right] \\ 
%& = \mathbb{E}_{\xi_{k}}\left[\left\langle \grad F_k(x^{k}), \grad F_k(x^{k}, \xi_{k} ) \right\rangle \right]-  \mathbb{E}_{\xi_{k}}\left[\left\langle \grad F_k(x^{k}), \grad F_k(x^{k}) \right\rangle \right] \\ 
%& =\|  \grad  F_k(x^{k})\|^{2}-\|  \grad  F_k(x^{k})\|^{2}
 %= 0 
\end{aligned}
\end{equation}
 Combining \eqref{theorem6}, \eqref{eq:variance} and \eqref{theorem466}, we have that
\begin{equation}\label{eq:expect1}
    \mathbb{E}_{\xi_k}  \left[ F_{k}(x^{k+1}) \right]\leq F_k(x^{k})-\left(\gamma_{k}-\frac{\ell_{k}}{2} \gamma_{k}^{2}\right)\|  \grad  F_k(x^{k})\|^{2}+\frac{\ell_{k}}{2} \gamma_{k}^{2}\sigma^2.
\end{equation}
Taking expectations with respect to all the previous realizations $\xi_1,\cdots,\xi_{k-1}$ on both sides of \eqref{eq:expect1} yields
\begin{equation}\label{eq:expect2}
    \mathbb{E}  \left[ F_{k}(x^{k+1}) \right]\leq  \mathbb{E}  \left[F_k(x^{k})\right]-\left(\gamma_{k}-\frac{\ell_{k}}{2} \gamma_{k}^{2}\right)  \mathbb{E}  \left[\|  \grad  F_k(x^{k})\|^{2}\right]+\frac{\ell_{k}}{2} \gamma_{k}^{2}\sigma^2.
\end{equation}
 It follows from Lemma \ref{lemma-1} that for all $x\in\mathcal{M}$ 
    \begin{equation*}
        F_{k+1}(x) \leq F_k(x) + \frac{1}{2} (\mu_k - \mu_{k+1}) \frac{\mu_k}{\mu_{k+1}} \| \nabla h_{\mu_k}(\mathcal{A}x)  \|^2 \leq F_k(x) + (\mu_k - \mu_{k+1}) \ell_h^2.
    \end{equation*}
    We set $x = x^{k+1}$ and take expectations on both sides for the above inequality, and substitute it into \eqref{eq:expect2} to obtain
    \begin{equation*}
    \begin{aligned}
     \mathbb{E}  \left[ F_{k+1}(x^{k+1}) \right]\leq &    \mathbb{E}  \left[F_k(x^{k})\right]-\left(\gamma_{k}-\frac{\ell_{k}}{2} \gamma_{k}^{2}\right)  \mathbb{E}  \left[ \| \grad F_k(x^{k})\|^{2}\right] \\
     & +\frac{\ell_{k}}{2} \gamma_{k}^{2}\sigma^2 + \left(\mu_k - \mu_{k+1}\right) \ell_h^2.
    \end{aligned}
    \end{equation*}
Summing up the above inequalities and re-arranging the terms yields
\begin{equation}\label{theorem555}
\begin{aligned}
& \sum_{k=1}^{K}\left(\gamma_{k}-\frac{\ell_{k}}{2} \gamma_{k}^{2}\right)  \mathbb{E}  \left[ \|  \grad  F_k(x^{k})\|^{2} \right]\\
  & \leq F_1(x^{1})-F_{K+1}(x^{K+1})+\frac{\sigma^2}{2} \sum_{k=1}^{K} \ell_{k}\gamma_{k}^{2}+\sum_{k=1}^{K}\left(\mu_k - \mu_{k+1}\right) \ell_h^2 \\
 & \leq \omega_2 +\frac{\sigma^2}{2} \sum_{k=1}^{K} \ell_{k}\gamma_{k}^{2},
\end{aligned}
\end{equation}
%where the first inequality follows from the fact that $ F_{K+1}(x^{K+1}) \geq  F^{*}$.
where $\omega_2$ is defined in \eqref{eq:omega}. The definition of $x^R$ in Algorithm \ref{alg:stochastic} implies that
\begin{equation}\label{theorem8}
\begin{aligned}
\mathbb{E}\left[\|\grad  F_R(x^{R})\|^{2}\right] =\frac{\sum_{k=1}^{K}\left(2 \gamma_{k}-\ell_{k} \gamma_{k}^{2}\right) \mathbb{E}\|\grad  F_k(x^{k})\|^{2}}{\sum_{k=1}^{K}\left(2 \gamma_{k}-\ell_{k} \gamma_{k}^{2}\right)}.
\end{aligned}
\end{equation}
Dividing both sides of \eqref{theorem555} by $ \sum_{k=1}^{K}\left(\gamma_{k}-\ell_{k} \gamma_{k}^{2} / 2\right)$ and substituting it into \eqref{theorem8} yields
\begin{equation}\label{theorem9}
\begin{aligned}
 \mathbb{E}\left[\|\grad  F_R(x^{R})\|^{2}\right]
 & \leq \frac{1}{\sum_{k=1}^{N}\left(2 \gamma_{k}-\ell_{k} \gamma_{k}^{2}\right)}\left[{2\omega_2}+\sigma^{2} \sum_{k=1}^{N} \ell_{k}\gamma_{k}^{2}\right],
 \end{aligned}
\end{equation}
which clearly implies \eqref{theorem2}.%and completes the proof.
\end{proof} 

To estimate the righthand side  of \eqref{theorem2}, we need the following lemma. 
\begin{lemma}\label{lemma:k-inequality}
 For any positive integer $K>0$, it follows that
  \begin{equation}
      \sum_{k=1}^K k^{-3/5}\geq \frac{5}{2} K^{\frac{2}{5}}, ~~\sum_{k=1}^K k^{-1} \leq \ln(3K).
  \end{equation}
\end{lemma}

\begin{proof}
  Given $k \geq 1$, it holds that that $k^{-3/5} \geq \int_{k}^{k+1} x^{-3/5}$. Then,
\begin{equation*}
\begin{aligned}
    \sum_{k=1}^K k^{-3/5} & \geq \sum_{k=1}^K \int_{k}^{k+1} x^{-3/5} dx  = \int_1^{K+1} x^{-3/5} dx = \frac{5}{2}\left((K+1)^{\frac{2}{5}} - 1\right) \geq \frac{5}{2} K^{\frac{2}{5}}. 
\end{aligned}
\end{equation*}
Similarly, we can obtain that
\begin{equation*}
    \begin{aligned}
    \sum_{k=1}^K k^{-1} &= 1 +  \sum_{k=2}^K k^{-1} \leq 1+ \sum_{k=2}^K \int_{k-1}^{k}  x^{-1} dx  =1+ \int_1^{K} x^{-1} dx = 1+ \ln(K)\leq \ln(3K).
\end{aligned}
\end{equation*}
We complete the proof.
\end{proof}

Based on Theorem \ref{thm:ssumptionA1} and Lemma \ref{lemma:k-inequality}, we obtain the main convergence results.
\begin{corollary}\label{corollary:stochastic}
Suppose that Assumption \ref{assum} and \ref{assum:stochatis} hold.   Let $\omega_1,\omega_2,\omega_3$ be constants given in \eqref{eq:omega} and $ \omega_4 = 2\rho \alpha^2 \|\mcA\|^2. $ Denote $\omega(K): = 4\omega\omega_1^{-3}  -  \omega^2\omega_4^{-5} \ln(3K){K^{-2/5}}$
and suppose  $\omega(K) \geq 0$. Let $\omega = 2\omega_1^{-1}\omega_4^3,$ and stepsize
\begin{equation}\label{theorem11}
\gamma_{k}= \frac{\omega}{\ell_{k}^3}, k=1, \ldots, K.
\end{equation}
 Then we have 
\begin{equation}\label{theorem12}
\mathbb{E}\left[\|\grad F_R(x^{R})\|^{2}\right] \leq \mathcal{B}_{K}:=\left(\omega_2\omega_1^{4}\omega_4^{-3} + 2\sigma^2 \omega_1^2 \omega_4^{-2}  \ln(3K)\right){K^{-2/5}},
\end{equation} If $\mcA$ is also surjective, then for $\hat{x}^R$ defined by $
    \hat{x}^R = x^R - \mathcal{A}^{\dagger}(\mathcal{A}x^R - \prox_{\mu_R h}(\mathcal{A}x^R)),$
 we have
\begin{equation}\label{theorem:stochastic:grad}
\begin{aligned}
   \mathbb{E}\left[ \mathrm{dist}\left(0,\mcP_{T_{x^R}\mcM}\left(\nabla f(\hat{x}^R) + \mcA^* \partial h(\mathcal{A} \hat{x}^R)\right)   \right) \right] & \leq C K^{-1/5} \\
    \|x^R - \hat{x}^R\| & \leq  \sigma_{\min}^{-1}(\mcA)  (2\rho)^{-1} R^{-1/5}\ell_h.
\end{aligned}
\end{equation}
where $C = \sqrt{\left((\omega_3 \ln(3K)   + \omega_2)\omega_1^{4}\omega_4^{-3}  + 2\sigma^2 \omega_1^2 \omega_4^{-2} \ln(3K)\right)}$.
\end{corollary} 

\begin{proof}
By the definition of $\mu_k$ in Algorithm \ref{alg:stochastic} and  \eqref{equ:L22}, we deduce that
\begin{equation}
    \begin{aligned}
k^{-1/5} \omega_1^{-1} =   k^{-1/5} \frac{  1}{ (\alpha^2 \ell_{\nabla f} +2G\beta) + 2\rho   \alpha^2 \|\mathcal{A}\|^2}  \leq   \frac{1}{\ell_k}  \leq k^{-1/5} \frac{1}{2\rho \alpha^2 \|\mcA\|^2} = k^{-1/5} \omega_4^{-1}.
    \end{aligned}
\end{equation}
% it is 
% \begin{equation}
%     \begin{aligned}
%   k^{-1/5} \omega_1^{-1}   \leq \ell_k^{-1}  \leq k^{-1/5} \omega_4^{-1}.
%     \end{aligned}
% \end{equation}
Combining  Lemma \ref{lemma:k-inequality} and \eqref{theorem11} yields
\begin{equation}\label{equ:lk:gamak}
 \sum_{k=1}^K    \ell_k \gamma_k^2 = \omega^2\sum_{k=1}^K    \ell_k^{-5} \leq  {\omega^2}{\omega_4^{-5}} \sum_{k=1}^K k^{-1} \leq {\omega^2}{\omega_4^{-5}} \ln(3K),
\end{equation}
and
\begin{equation}\label{equ:gamak1}
 \sum_{k=1}^K     \gamma_k = \omega\sum_{k=1}^K    \ell_k^{-3} \geq
 \omega\omega_1^{-3}   \sum_{k=1}^K k^{-3/5} \geq \frac{5}{2}\omega\omega_1^{-3} K^{\frac{2}{5}}. 
\end{equation}
 It is clear that 
\begin{equation}
    0<\gamma_k = \frac{\omega}{\ell_k^3} = \frac{2\omega_4^3}{\omega_1}\ell_k^{-3} \leq k^{-3/5}\frac{2}{\omega_1} \leq k^{-1/5}\frac{2}{\omega_1} \leq \frac{2}{\ell_k}.
\end{equation}
Thus  $2\gamma_k - \ell_k\gamma_k^2 >0$ for all $k>0$.
This implies that
\begin{equation}\label{theorem13}
\begin{aligned}
\frac{2\omega_2+\sigma^{2} \sum_{k=1}^{K} \ell_k\gamma_{k}^2}{\sum_{k=1}^{K}\left(2 \gamma_{k}-\ell_{k} \gamma_{k}^{2}\right)} & \leq  \frac{2\omega_2+\sigma^{2}\omega^2\omega_4^{-5}\ln(3K)}{5\omega\omega_1^{-3} K^{2/5} -  \omega^2\omega_4^{-5}\ln(3K) } \\
& = \frac{\left(2\omega_2+\sigma^{2}\omega^2\omega_4^{-5} \ln(3K)\right){K^{-2/5}}}{5\omega\omega_1^{-3}  -  \omega^2\omega_4^{-5}{\ln(3K)}{K^{-2/5}} } \\
& =\frac{\left(2\omega_2+4\sigma^{2}\omega_1^{-2}\omega_4 \ln(3K)\right) K^{-2/5}}{\omega\omega_1^{-3}  +\omega(K)} \\
& \leq \frac{\left(2\omega_2+4\sigma^{2}\omega_1^{-2}\omega_4 \ln(3K)\right) K^{-2/5}}{2\omega_1^{-4}\omega_4^{3}  } \\
& = \left(\omega_2\omega_1^{4}\omega_4^{-3}  + 2\sigma^2 \omega_1^2 \omega_4^{-2} \ln(3K)\right) \times {K^{-2/5}},
\end{aligned}
\end{equation}
where the first inequality utilizes \eqref{equ:lk:gamak} and \eqref{equ:gamak1}, the second inequality is due to that $\omega(K) >0$. Combining with \eqref{theorem2} implies \eqref{theorem12}.

If $\mcA$ is also surjective, the definition of $\hat{x}^R$ implies that
%  \begin{equation}\label{eq:xstar3}
%   x^R -  \hat{x}^R =  \mathcal{A}^{\dagger}(\mathcal{A}x^R - \prox_{\mu_R h}(\mathcal{A}x^R)),
% \end{equation}
%  the operator norm of $\mcA^{\dagger}$ is bounded by $\sigma_{\min}^{-1}(\mcA)$ and $h$ is $\ell_h$-Lipschitz continuous, we have
 $$ \begin{aligned}
     \|x^R -\hat{x}^R\| \leq  \sigma_{\min }^{-1}(\mcA)\|\mcA x^R-z^R\|  \leq \frac{\sigma_{\min}^{-1}(\mcA) (2\rho)^{-1} \ell_h}{R^{1/5}}.
    \end{aligned}$$
Furthermore, \eqref{eq:composite2} implies that 
 
     \begin{equation}\label{eq:composite211}
    \begin{aligned}
\mathrm{dist}\left(0,\mcP_{T_{x^R}\mcM}(\nabla f(\hat{x}^R) + \mcA^* \partial h(\mcA \hat{x}^R))   \right)
 \leq    \frac{\ell_{\nabla f}\sigma_{\min}^{-1}(\mcA) (2\rho)^{-1} \ell_h}{R^{1/5}}+\| \grad F_R(x^R)\|.
\end{aligned}
\end{equation}
 Note that \begin{equation}\label{equ:gamak11}
        \begin{aligned}
        \left(\gamma_{k}-\frac{\ell_{k}}{2} \gamma_{k}^{2}\right) \leq \omega \ell_k^{-3} -\frac{1}{2} \omega^2 \ell_k^{-5} \leq \omega \omega^{-3} k^{-3/5}.
        \end{aligned}
        \end{equation}
 This implies that
  \begin{equation}
  \begin{aligned}
     & \sum_{k=1}^K  \left(\gamma_{k}-\frac{\ell_{k}}{2} \gamma_{k}^{2}\right) \mathrm{dist}^2\left(0,\mcP_{T_{x^k}\mcM}\left(\nabla f(\hat{x}^k) + \mcA^* \partial h(\mcA \hat{x}^k)\right)   \right) \\
      & \leq  2\sum_{k=1}^K \left(\gamma_{k}-\frac{\ell_{k}}{2} \gamma_{k}^{2}\right)\frac{(\ell_{\nabla f}\sigma_{\min}^{-1}(\mcA) (2\rho)^{-1} \ell_h)^2}{k^{2/5}} +  2\sum_{k=1}^K \left(\gamma_{k}-\frac{\ell_{k}}{2} \gamma_{k}^{2}\right) \| \grad F_k(x^k)\|^2\\
      & \leq 2\omega_3 \sum_{k=1}^K \frac{1}{k}   + 2\omega_2 +\sigma^2 \sum_{k=1}^{K} \ell_{k}\gamma_{k}^{2}  \leq 2\omega_3 \ln(3K)   +2 \omega_2 +\sigma^2 \sum_{k=1}^{K} \ell_{k}\gamma_{k}^{2}.\\
  \end{aligned}
  \end{equation}
  Similar to the above analysis, one can obtain that
  \begin{equation}
  \begin{aligned}
           & \mathbb{E}\left[\mathrm{dist}^2\left(0,\mcP_{T_{x^R}\mcM}\left(\nabla f(\hat{x}^R) + \mcA^* \partial h(\mcA \hat{x}^R)\right)  \right)\right]  \\
           & \leq  \left((\omega_3 \ln(3K)   + \omega_2)\omega_1^{4}\omega_4^{-3}  + 2\sigma^2 \omega_1^2 \omega_4^{-2} \ln(3K)\right) \times {K^{-2/5}},
           %\\
          % &\revise{ \leq  \left((\omega_3 
          % \ln(3K)   + \omega_2)\omega_1^{4}\omega_4^{-3}  + 2\sigma^2 \omega_1^2 \omega_4^{-2} \ln(3K)\right) \times {K^{-2/5}}},
  \end{aligned}
         \end{equation}
    which implies that \eqref{theorem:stochastic:grad} holds. The proof is completed. 
\end{proof} 

\begin{remark}
In Corollary \ref{corollary:stochastic}, we assume that   $\omega(K)\geq 0$, it is reasonable because that when $K$ is large enough  ${\ln(3K)}{K^{-2/5}}$ will tend to zero.   Corollary \ref{corollary:stochastic} claims that, given $\epsilon>0$, the total  iterations  performed by Algorithm \ref{alg:stochastic} for   is $\tilde{\mathcal{O}}(\epsilon^{-5})$. 

\end{remark}
\subsection{Riemannian smoothing stochastic gradient method with epochs}

We give a variant of Algorithm \ref{alg:stochastic} in which the steps are organized into a series of epochs, each of which is twice as long as the one before, which is described in Algorithm \ref{alg:stochastic-epoch}.

\begin{algorithm}
\caption{Riemannian smoothing stochastic gradient method with epochs for   \eqref{eq:rsg-22}.} \label{alg:stochastic-epoch}
  \textbf{Input:}{$x^1\in\mathcal{M}$ and tolerance  $\epsilon>0$; }\\
  \For{$l = 0,1,2,\cdots$}
  {
 % \revise{Set $S_l = \infty, k_l = 2^l$.} \\
  \For{$k = 2^l,2^l+1,\cdots,2^{l+1}-1 $}{
 Sample $\xi_k \sim \mathcal{D}$, set $\mu_k = (2\rho)^{-1}k^{-1/5} $ and set $\ell_k$ as in \eqref{equ:L22}. \\
 Compute the stochastic gradient $\grad  F_k(x^{k},\xi_k )$ via \eqref{eq:riestochas}.\\
 Update $x^{k+1}$ via
\begin{equation}\label{eq:rsg21}
x^{k+1} = \mathcal{R}_{x^k}( - \gamma_k\grad  F_k(x^{k},\xi_k )).
\end{equation}
% \If{$\|\operatorname{grad} F_{k+1}(x^{k+1})\| \leq S_{l}$}{
%  $S_l \leftarrow \|\operatorname{grad} F_{k+1}(x^{k+1})\| $, \revise{$k_l \leftarrow k+1$};\\
% \If{$S_{l} \leq \epsilon$ and $\|\mathcal{A} x^{k+1}-\operatorname{prox}_{\mu_{k+1} h}(\mathcal{A} x^{k+1})\|  \leq \epsilon $}{
% break;}
% }
\textbf{end}}
Sample $R_l\in\{2^l,\cdots,2^{l+1}-1\}$ via $\mathrm{Prob}(R_l = k) = \frac{2 \gamma_{k}-\ell_{k} \gamma_{k}^{2}}{\sum_{k=2^l}^{2^{l+1}-1}\left(2 \gamma_{k}-\ell_{k} \gamma_{k}^{2}\right)}$. 
\textbf{end} }
%\KwOut{ \revise{$x^{R}$}}
\end{algorithm}

\begin{theorem}\label{thm:sto:epoch}
Suppose that Assumption \ref{assum} and \ref{assum:stochatis} hold. Let $\omega_1,\omega_2,\omega_3,\omega_4,\omega$ be constants given in Corollary \ref{corollary:stochastic}. The iteration sequence $\{x^k\}$ is produced by  Algorithm \ref{alg:stochastic-epoch} for solving problem \eqref{eq:rsg-22}.  Assume that $\mcA$ is surjective. Then, for some $k^{\prime}=O(\epsilon^{-5})$, we have that
	\begin{equation}\label{theorem:stochastic:grad2}
\begin{aligned}
   \mathbb{E}\left[ \mathrm{dist}\left(0,\mcP_{T_{x^{k^{\prime}}}\mcM}\left(\nabla f(\hat{x}^{k^{\prime}}) + \mcA^* \partial h(\mathcal{A} \hat{x}^{k^{\prime}})\right)   \right)\right]  \leq \epsilon, ~~
    \|x^{k^{\prime}} - \hat{x}^{k^{\prime}}\| & \leq  \epsilon,
\end{aligned}
\end{equation}
where $\hat{x}^{k^{\prime}}: = x^{k^{\prime}} - \mathcal{A}^{\dagger}(\mathcal{A}x^{k^{\prime}} - \prox_{\mu_k h}(\mcA x^{k^{\prime}}) )$.
\end{theorem}

\begin{proof}
 Given \eqref{theorem555}, we can leverage the monotonicity of $\left\{\grad F_{k}\left(x_{k}\right)\right\}$ and omit nonnegative terms. This allows us establishing the following inequality: 
\begin{equation}\label{theorem77}
\sum_{k=2^{l}}^{2^{l+1}-1}\left(\gamma_{k}-\frac{\ell_{k}}{2} \gamma_{k}^{2}\right) \mathbb{E}_{\xi}\| \grad F_k(x^{k})\|^{2} \leq  \omega_2 +\frac{\sigma^{2}}{2} \sum_{k=2^{l}}^{2^{l+1}-1} \ell_{k} \gamma_{k}^{2}.
\end{equation}
Similar to Lemma \ref{lemma:k-inequality}, one can easily show that
\begin{equation} \label{equ:stochastic:k}
    \sum_{k=2^{l}}^{2^{l+1}-1} k^{-3/5} \geq \frac{1}{2}(2^{l})^{2 / 5},~~ \sum_{k=2^{l}}^{2^{l+1}-1} k^{-1} \leq \ln(4) \leq 2.
\end{equation}
Additionally, based on the definition of $\hat{x}^{R_l}$ in Corollary \ref{corollary:stochastic}, it can be inferred that
 $$ \begin{aligned}
     \|x^{R_l} -\hat{x}^{R_l}\| \leq  \sigma_{\min }^{-1}(\mcA)\|\mcA x^{R_l}-z^{R_l}\|  \leq \frac{\sigma_{\min}^{-1}(\mcA) (2\rho)^{-1} \ell_h}{{R_l}^{1/5}}  \leq \frac{\sigma_{\min}^{-1}(\mcA) (2\rho)^{-1} \ell_h}{(2^l)^{1/5}}.
    \end{aligned}$$
As in \eqref{eq:composite2}, one can obtain that     \begin{equation}\label{eq:composite2111}
    \begin{aligned}
\mathrm{dist}\left(0, \mcP_{T_{x^{R_l}}\mcM}(\nabla f(\hat{x}^{R_l}) + \mcA^* \partial h(\mcA \hat{x}^{R_l}))   \right)
 \leq    \frac{\ell_{\nabla f}\sigma_{\min}^{-1}(\mcA) (2\rho)^{-1} \ell_h}{{R_l}^{1/5}}+\| \grad F_{R_l}(x^{R_l})\|.
\end{aligned}
\end{equation}
 This implies that
  \begin{equation}
  \begin{aligned}
     & \sum_{k=2^l}^{2^{l+1}-1} \left(\gamma_{k}-\frac{\ell_{k}}{2} \gamma_{k}^{2}\right) \mathrm{dist}^2\left(0, \mcP_{T_{x^k}\mcM}\left(\nabla f(\hat{x}^k) + \mcA^* \partial h(\mcA \hat{x}^k)\right)  \right) \\
      & \leq  2\sum_{k=2^l}^{2^{l+1}-1} \left(\gamma_{k}-\frac{\ell_{k}}{2} \gamma_{k}^{2}\right)\frac{(\ell_{\nabla f}\sigma_{\min}^{-1}(\mcA) (2\rho)^{-1} \ell_h)^2}{k^{2/5}} +  2\sum_{k=2^l}^{2^{l+1}-1} \left(\gamma_{k}-\frac{\ell_{k}}{2} \gamma_{k}^{2}\right) \| \grad F_k(x^k)\|^2\\
      & \leq 2\omega_3 \sum_{k=2^l}^{2^{l+1}-1} \frac{1}{k}   + 2\omega_2 +\sigma^2 \sum_{k=2^l}^{2^{l+1}-1} \ell_{k}\gamma_{k}^{2}  \leq 4\omega_3    +2 \omega_2 +\sigma^2 \sum_{k=2^l}^{2^{l+1}-1} \ell_{k}\gamma_{k}^{2},\\
  \end{aligned}
  \end{equation}
 where the second inequality utilizes \eqref{equ:gamak11} and the last inequality utilizes \eqref{equ:stochastic:k}.  Similar to the above analysis, one can obtain that
  \begin{equation}
  \begin{aligned}
           & \mathbb{E}\left[\mathrm{dist}^2 \left(0,\mcP_{T_{x^{R_l}}\mcM}\left(\nabla f(\hat{x}^{R_l}) + \mcA^* \partial h(\mcA \hat{x}^{R_l})\right)  \right)\right]  \\
           & \leq  \left((2\omega_3    + \omega_2)\omega_1^{4}\omega_4^{-3}  + 2\sigma^2 \omega_1^2 \omega_4^{-2} \ln(3K)\right) \times {(2^l)^{-2/5}},
           %\\
          % &\revise{ \leq  \left((\omega_3 
          % \ln(3K)   + \omega_2)\omega_1^{4}\omega_4^{-3}  + 2\sigma^2 \omega_1^2 \omega_4^{-2} \ln(3K)\right) \times {K^{-2/5}}},
  \end{aligned}
  \end{equation}
 With the considerations made in the previous proof, we can choose $l$ to be the smallest positive integer such that
	$$
	2^{l+1} \geq 2 \max \left\{8\sqrt{\omega_1^{3}(\omega_2+\omega_3)^{3}}, \sigma_{\min }^{-3}(\mathcal{A}) \ell_{h}^{3}(2 \rho)^{-3}\right\} \epsilon^{-5}.
	$$
	The claim then holds for some $k^{\prime} \leq 2^{l+1}$. The proof is completed.
\end{proof}

\section{Numerical results}\label{sec:num}
In this section, some numerical experiments are presented to evaluate the performance of our four algorithms, referred to as R-full, R-full-epoch, R-stochastic, and R-stochastic-epoch, correspond to algorithms \ref{alg:full} to \ref{alg:stochastic-epoch} respectively. We compare our algorithms to the existing methods including ManPG \cite{chen2020proximal} and its adaptive version (ManPG-A), and the Riemannian subgradient method (R-sub) in \cite{li2021weakly}. At the $k$-th iteration,  the Riemannian subgradient method gets $x^{k+1}$ via 
$$
x^{k+1} = \mathcal{R}_{x^k}(-\gamma_k \tilde{\nabla}_R F(x^k)), \tilde{\nabla}_R F(x^k) \in \partial_R F(x^k), 
$$
where $\partial_R F(x^k)$ denote the Riemannian subgradient of $F$ at $x^k$ and $\gamma_k$ is a step size.  Our four algorithms are terminated  if
\begin{equation}\label{kkt stop1}
    \max\left\{\|\grad F_k(x^k)\| ,\|\mathcal{A}x^k - \prox_{\mu_k h}(\mathcal{A}x^k)\|\right\} \leq \mbox{tol},
\end{equation}
where ``tol'' is a given accuracy tolerance.  
The ManPG and ManPG-A algorithms are terminated if
\begin{equation}\label{kkt stop3}
     \|E(\Lambda^k)\|_2 \leq \mbox{tol},
\end{equation}
where $E(\Lambda^k)$ is given by  \cite{chen2020proximal}. In the following experiments, we will first run the ManPG algorithm, and terminate it when either condition \eqref{kkt stop3} is satisfied or the maximum iteration steps of 10,000 are reached. The obtained function value is denoted as $F_M$. For our four algorithms, we terminate them when either of the following three conditions is met:
\begin{enumerate}
    \item[(1)] the criterion \eqref{kkt stop1} is hit with the given tolerance;
    \item[(2)] the maximum iteration steps of 1000 are reached;
    \item[(3)] the objective function value satisfies $F(x^k)\leq F_M - 10^{-10}$.
\end{enumerate}
For the R-sub, we terminate it when either the objective function value satisfies $F(x^k)\leq F_M + 10^{-10}$ or the maximum iteration steps of 10,000 are reached.

\subsection{Sparse principal component analysis}
In this subsection, we compare those algorithms on sparse principal component analysis (SPCA) problems.
Given a data set $\{b_1,\cdots, b_m\}$ where $b_i\in\mathbb{R}^{n\times 1}$, the SPCA problem is
\begin{equation}\label{spca}
\begin{aligned}
  \min_{X\in\mathbb{R}^{n\times r}}   \sum_{i=1}^{m}\|b_i - XX^Tb_i\|_2^2 + \mu \|X\|_1, ~ ~ \mbox{s.t. }~~  X^TX = I_r,
  \end{aligned}
\end{equation}
where $\mu > 0$ is a regularization parameter. Let $B = [b_1,\cdots, b_m ]^{T}\in\mathbb{R}^{m\times n}$, problem \eqref{spca} can be rewritten as: %we have %obtain the following equivalent problem
\begin{equation}\label{spcaM}
  \begin{aligned}
     \min_{X\in\mathbb{R}^{n\times r}}   -\mathrm{tr}(X^TB^TBX) + \mu \|X\|_1,     ~~  \mbox{s.t. }~~  X^TX = I_r.
  \end{aligned}
\end{equation}
Here, the constraint consists of the Stiefel manifold $\texttt{St}(n,r):=\{X\in\mathbb{R}^{n\times r}~:~X^\top X = I_r\}$. The tangent space of $\texttt{St}(n,r)$ is defined by $T_{X}\texttt{St}(n,r) = \{\eta\in \mathbb{R}^{n\times r}~:~X^\top \eta + \eta^\top X = 0\}$. Given any $U\in\mathbb{R}^{n\times r}$, the projection of $U$ onto $T_{X}\texttt{St}(n,r)$ is $\mathcal{P}_{T_{X}\texttt{St}(n,r)}(U) = U - X \frac{U^\top X + X^\top U}{2}$ \cite{AbsMahSep2008}. In our experiment, the data matrix $B\in\mathbb{R}^{m\times n}$ is produced by MATLAB function $\texttt{randn}(m, n)$, in which all entries of $B$ follow the standard Gaussian distribution. We shift the columns of $B$ such that they have zero mean, and finally the column vectors are normalized. We use the polar decomposition as the retraction mapping.%is generated via the following two settings:

We first compare the performance of the proposed four algorithms on the SPCA problem. We set $m=5000, n=200,$ and $\mu=0.4$. For the R-stochastic and R-stochastic-epoch algorithms, we partition the 50,000 samples into 100 subsets, and in each iteration, we sample one subset. The tolerance is set $\mathrm{tol} = 10^{-8}*n*r$.  Figure \ref{fig:perf_spca} presents the results of the four algorithms for fixed $m=5000,n=200,\mu=0.4$ and varying $r=5$ and $r=10$. The horizontal axis represents CPU time, while the vertical axis represents the objective function value gap: $F(x^k) - F_M$, where $F_M$ is given by the ManPG. The results indicate that the random version of the smoothing algorithm outperforms the deterministic version. Moreover, among the deterministic versions, Algorithm \ref{alg:full-epoch} (i.e., R-full-epoch) performs better than Algorithm \ref{alg:full} (i.e., R-full). In conclusion, the usages of epochs and stochastic estimations of the gradients lead to faster convergence.

We test the performance of all algorithms for solving the SPCA problems with different $n$, $r$ and sparsity parameter $\mu$, where $n \in\{ 200,300,500\}; r \in\{ 10, 20,30,40,50\}; \mu \in\{0.4,0.6,0.8\} $. As shown in Figure \ref{fig:perf_spca}, the R-stochastic demonstrates superior performance compared to other algorithms. Therefore, in this study, we only compare this algorithm with others, including ManPG, ManPG-A, and R-sub. Table \ref{tab:cms_ss} reports the computational results of our algorithm compared with other algorithms. From the table, it can be observed that R-sub does not converge before reaching the maximum iterations. The R-stochastic algorithm, ManPG, and ManPG-A achieve the same objective function value, and R-stochastic exhibits shorter computation time compared to other algorithms. This may be attributed to the fact that R-stochastic is a stochastic algorithm, while ManPG and ManPG-A are deterministic algorithms.

  \begin{figure}[!htb]
\centering
\includegraphics[width=0.45\textwidth]{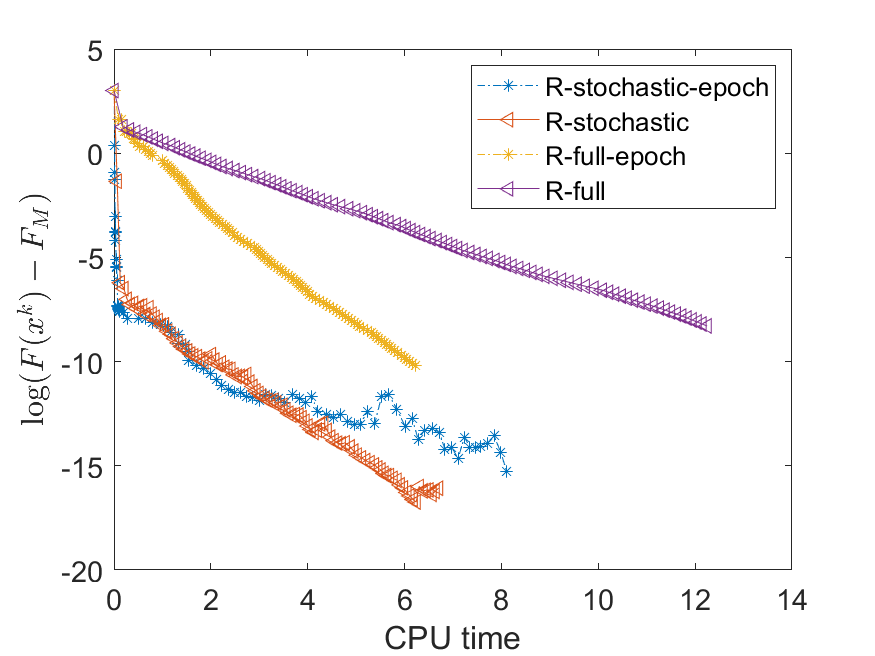}
\includegraphics[width=0.45\textwidth]{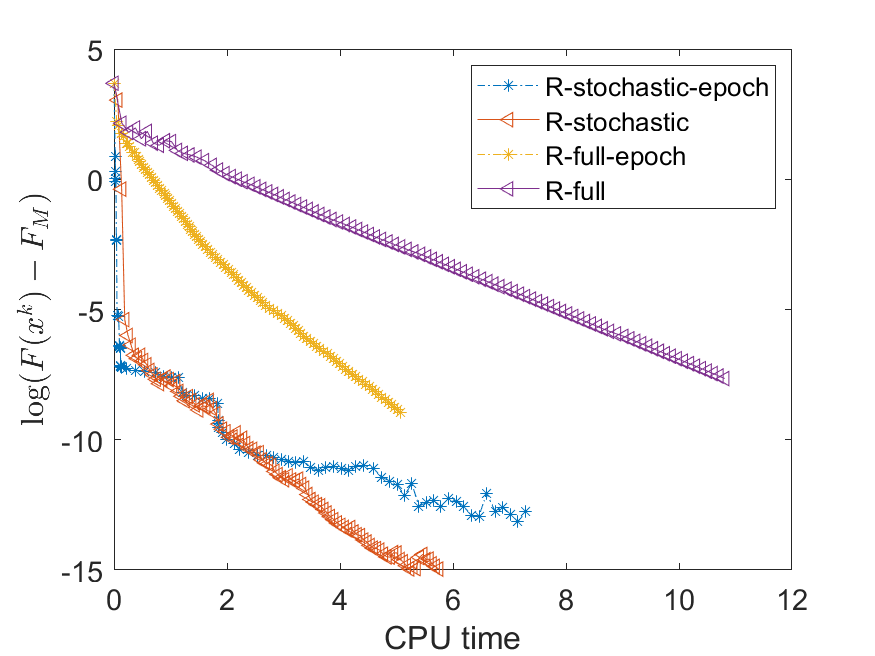}
\caption{The performance profiles of our four algorithms on SPCA with $m=5000,n=200$ and $\mu = 0.4$, left: $r=5$; right: $r=10$.}\label{fig:perf_spca}
\end{figure}

% \begin{itemize}
%   \item[(1)] Data matrix $B\in\mathbb{R}^{m\times n}$ is produced by MATLAB function $randn(m, n)$, which all entries of $B$ follow the standard Gaussian distribution. We  shift the columns of $B$ such that they have $0$-mean, and finally the column-vectors are normalized.
%   \item[(2)] 
%   Data matrix $B\in\mathbb{R}^{m\times n}$ is selected from real data. In those real data, ``Arabidopsis'',``Leukemia'' are the gene expression data selected from \cite{li2010inexact}, ``Staunton'',``Ross'' are the NCI 60 data selected from \cite{culhane2003cross}, and ``realEQTL'' is the yeast eQTL data selected from \cite{zhu2008integrating}.   
% \end{itemize}

\begin{footnotesize}
\setlength{\tabcolsep}{0.6pt}
%\footnotesize
\centering
\begin{longtable}{|c|cccc|cccc|}
\caption{The results of our algorithm compared with other algorithms on the sparse PCA problem with random data. The row labeled ``Percentage'' indicates the ratio of the algorithm is dominant in CPU time among all compared algorithms. The values of bold mean the optimal performance among all algorithms.}\label{tab:cms_ss}\\ 
\hline
\multirow{2}{1cm}{$(n,r,\mu)$} & \multicolumn{4}{|c|}{Objective function value} & \multicolumn{4}{|c|}{CPU time} \\ %\cline{2-9}
        &  R-stochastic &  ManPG    & ManPG-A & R-sub  &  R-stochastic &  ManPG    & ManPG-A & R-sub \\ \hline
\endfirsthead
\hline
\multirow{2}{*}{$(n,r,\mu)$} & \multicolumn{4}{c|}{Objective function value} & \multicolumn{4}{c|}{CPU time} \\ %\cline{2-9}
        &  R-stochastic &  ManPG    & ManPG-A & R-sub  &  R-stochastic &  ManPG    & ManPG-A & R-sub \\ \hline
\endhead
\hline
\endfoot

  (200,10,0.04) & -1.988e+1 & -1.988e+1 & -1.988e+1 & -1.018e+1 &  0.21 & 0.36 & \textbf{0.18} & 2.65\\ \hline
 (200,10,0.06) & -1.823e+1 & -1.823e+1 & -1.823e+1 & -8.224e+0 &  0.35 & 0.58 & \textbf{0.28} & 2.66\\ \hline
 (200,10,0.08) & -1.670e+1 & -1.669e+1 & -1.669e+1 & -6.512e+0 &  0.41 & 0.48 & \textbf{0.25} & 2.67\\ \hline
 (200,20,0.04) & -3.722e+1 & -3.722e+1 & -3.722e+1 & -2.004e+1 &  4.69 & 3.84 & \textbf{2.04} & 5.88\\ \hline
 (200,20,0.06) & -3.417e+1 & -3.417e+1 & -3.415e+1 & -1.579e+1 &  \textbf{2.50} & 5.75 & 2.76 & 5.60\\ \hline
 (200,20,0.08) & -3.051e+1 & -3.051e+1 & -3.051e+1 & -1.095e+1 &  \textbf{1.19} & 2.83 & 1.28 & 6.01\\ \hline
 (200,30,0.04) & -5.284e+1 & -5.284e+1 & -5.285e+1 & -2.903e+1 &  \textbf{1.37} & 6.38 & 3.41 & 9.62\\ \hline
 (200,30,0.06) & -4.806e+1 & -4.822e+1 & -4.814e+1 & -2.291e+1 &  7.00 & 12.31 & \textbf{4.56} & 10.03\\ \hline
 (200,30,0.08) & -4.383e+1 & -4.383e+1 & -4.400e+1 & -1.662e+1 &  \textbf{0.83} & 6.27 & 4.74 & 10.37\\ \hline
 (200,40,0.04) & -6.618e+1 & -6.618e+1 & -6.613e+1 & -3.715e+1 &  10.04 & 14.94 & \textbf{4.93} & 14.92\\ \hline
 (200,40,0.06) & -6.071e+1 & -6.078e+1 & -6.078e+1 & -2.969e+1 &  10.73 & 14.02 & \textbf{6.53} & 14.35\\ \hline
 (200,40,0.08) & -5.528e+1 & -5.527e+1 & -5.528e+1 & -2.093e+1 &  \textbf{3.89} & 16.18 & 7.18 & 14.84\\ \hline
 (200,50,0.04) & -7.904e+1 & -7.908e+1 & -7.905e+1 & -4.450e+1 &  13.75 & 31.19 & \textbf{10.10} & 22.85\\ \hline
 (200,50,0.06) & -7.193e+1 & -7.192e+1 & -7.189e+1 & -3.480e+1 &  \textbf{2.59} & 27.55 & 7.27 & 23.98\\ \hline
 (200,50,0.08) & -6.542e+1 & -6.543e+1 & -6.553e+1 & -2.504e+1 &  16.40 & 22.24 & \textbf{13.43} & 23.45\\ \hline
 (300,10,0.04) & -2.351e+1 & -2.351e+1 & -2.351e+1 & -1.160e+1 &  \textbf{0.12} & 0.43 & 0.23 & 4.25\\ \hline
 (300,10,0.06) & -2.152e+1 & -2.152e+1 & -2.152e+1 & -9.712e+0 &  2.99 & 0.76 & \textbf{0.36} & 4.12\\ \hline
 (300,10,0.08) & -1.948e+1 & -1.951e+1 & -1.951e+1 & -7.184e+0 &  3.84 & 1.05 & \textbf{0.49} & 4.18\\ \hline
 (300,20,0.04) & -4.422e+1 & -4.428e+1 & -4.428e+1 & -2.337e+1 &  4.66 & 2.45 & \textbf{1.46} & 7.17\\ \hline
 (300,20,0.06) & -4.038e+1 & -4.036e+1 & -4.036e+1 & -1.803e+1 &  \textbf{0.95} & 3.43 & 1.75 & 7.19\\ \hline
 (300,20,0.08) & -3.634e+1 & -3.634e+1 & -3.640e+1 & -1.412e+1 &  \textbf{1.67} & 3.87 & 2.80 & 7.33\\ \hline
 (300,30,0.04) & -6.340e+1 & -6.340e+1 & -6.336e+1 & -3.347e+1 &  \textbf{3.77} & 12.17 & 3.90 & 14.65\\ \hline
 (300,30,0.06) & -5.702e+1 & -5.702e+1 & -5.709e+1 & -2.685e+1 &  10.32 & 10.34 & \textbf{5.87} & 14.33\\ \hline
 (300,30,0.08) & -5.171e+1 & -5.171e+1 & -5.171e+1 & -1.878e+1 &  \textbf{1.54} & 9.67 & 4.48 & 14.66\\ \hline
 (300,40,0.04) & -8.068e+1 & -8.068e+1 & -8.074e+1 & -4.385e+1 &  \textbf{1.19} & 12.95 & 10.10 & 19.33\\ \hline
 (300,40,0.06) & -7.266e+1 & -7.266e+1 & -7.269e+1 & -3.235e+1 &  \textbf{8.45} & 18.43 & 8.51 & 19.87\\ \hline
 (300,40,0.08) & -6.595e+1 & -6.594e+1 & -6.613e+1 & -2.479e+1 &  \textbf{3.25} & 13.16 & 11.99 & 19.56\\ \hline
 (300,50,0.04) & -9.606e+1 & -9.605e+1 & -9.608e+1 & -5.353e+1 &  \textbf{2.41} & 25.50 & 16.59 & 26.65\\ \hline
 (300,50,0.06) & -8.703e+1 & -8.724e+1 & -8.713e+1 & -4.047e+1 &  \textbf{16.81} & 52.02 & 20.25 & 26.74\\ \hline
 (300,50,0.08) & -7.826e+1 & -7.826e+1 & -7.804e+1 & -2.851e+1 &  \textbf{5.34} & 39.97 & 11.42 & 31.45\\ \hline
 (500,10,0.04) & -3.056e+1 & -3.056e+1 & -3.056e+1 & -1.692e+1 &  \textbf{0.62} & 1.55 & 0.85 & 10.42\\ \hline
 (500,10,0.06) & -2.757e+1 & -2.757e+1 & -2.757e+1 & -1.362e+1 &  1.43 & 1.40 & \textbf{0.69} & 10.44\\ \hline
 (500,10,0.08) & -2.473e+1 & -2.476e+1 & -2.476e+1 & -1.070e+1 &  8.15 & 1.16 & \textbf{0.67} & 10.79\\ \hline
 (500,20,0.04) & -5.820e+1 & -5.820e+1 & -5.817e+1 & -3.254e+1 &  \textbf{2.44} & 5.40 & 2.95 & 17.92\\ \hline
 (500,20,0.06) & -5.196e+1 & -5.196e+1 & -5.196e+1 & -2.543e+1 &  \textbf{0.68} & 2.90 & 1.68 & 18.26\\ \hline
 (500,20,0.08) & -4.649e+1 & -4.649e+1 & -4.652e+1 & -1.895e+1 &  \textbf{2.97} & 5.74 & 4.32 & 19.35\\ \hline
 (500,30,0.04) & -8.199e+1 & -8.199e+1 & -8.189e+1 & -4.672e+1 &  \textbf{4.53} & 20.92 & 6.13 & 31.88\\ \hline
 (500,30,0.06) & -7.443e+1 & -7.446e+1 & -7.450e+1 & -3.727e+1 &  16.30 & 9.26 & \textbf{6.35} & 29.98\\ \hline
 (500,30,0.08) & -6.724e+1 & -6.723e+1 & -6.729e+1 & -2.682e+1 &  \textbf{2.70} & 15.76 & 8.49 & 30.34\\ \hline
 (500,40,0.04) & -1.060e+2 & -1.060e+2 & -1.060e+2 & -5.875e+1 &  \textbf{7.79} & 30.92 & 20.98 & 38.18\\ \hline
 (500,40,0.06) & -9.557e+1 & -9.557e+1 & -9.554e+1 & -4.867e+1 &  \textbf{3.90} & 45.70 & 13.98 & 38.03\\ \hline
 (500,40,0.08) & -8.492e+1 & -8.492e+1 & -8.499e+1 & -3.362e+1 &  \textbf{3.91} & 36.98 & 21.05 & 37.90\\ \hline
 (500,50,0.04) & -1.271e+2 & -1.271e+2 & -1.272e+2 & -7.299e+1 &  \textbf{8.67} & 29.91 & 20.22 & 52.52\\ \hline
 (500,50,0.06) & -1.140e+2 & -1.140e+2 & -1.142e+2 & -5.714e+1 &  \textbf{4.21} & 39.19 & 29.72 & 52.59\\ \hline
 (500,50,0.08) & -1.025e+2 & -1.025e+2 & -1.024e+2 & -4.153e+1 &  \textbf{12.17} & 57.37 & 26.65 & 53.36\\ \hline
Percentage & & & & & \textbf{64.44\%} & 0 & 35.56\% & 0 \\ \hline

\end{longtable}
\end{footnotesize}

%   Let $\Lambda = \mbox{Diag}([\lambda_1,\cdots,\lambda_r])$ where $\lambda_1 = \cdots=\lambda_r = 0.9$, and  construct a sparse matrix $X\in\mathbb{R}^{n\times r}$ where the rows of $X$ are set to $S = \{1:\lfloor n/5\rfloor\}$, and the corresponding value at the nonzero coordinates are generated by normalizing random numbers drawn from the uniform distribution on finite set $\{-2, -1, 0, 1, 2\}$.  Then the covariance matrix $\Sigma_X$ is given by
%       \begin{equation}\nn
%         \Sigma_X = X\Lambda X^T + 0.01 \cdot EE^T,
%       \end{equation}
%       where $E$ denote the noise matrix. Finally,  each $b_i$ in data matrix $B$ is generated via
%       \begin{equation}\nn
%         b_i \sim \mathcal{N}(0,\Sigma_X), i = 1, \cdots, m.
%       \end{equation}

\subsection{Compressed modes in Physics}

 In physics, the compressed modes (CMs) problem   seeks spatially localized solutions of the independent-particle Schr\"{o}dinger equation:
\begin{equation}
  \hat{H}\phi(x) = \lambda \phi(x), ~~x\in\Omega,
\end{equation}
where $\hat{H} = -\frac{1}{2}\Delta$  and $\Delta$ is  a Laplacian operator. %Letting $H$ be a symmetric matrix formed by the discretization of $\hat{H}$.
Consider the $1$-D free-electron (FE) model  with $\hat{H} = -\frac{1}{2}\partial_{x^2}$ in our experiments.  By proper discretization, the CMs can be reformulated to
\begin{equation}\label{CMs}
  \begin{aligned}
     \min_{X\in\mathbb{R}^{n\times r}}  \mathrm{tr}(X^TH X) + \mu \|X\|_1, ~~ \mbox{s.t. } ~~  X^TX = I_r,
  \end{aligned}
\end{equation}
where $H$ is the discretized Schr\"{o}dinger operator, $\mu$ is a regularization parameter. The interesting readers are referred to \cite{ozolicnvs2013compressed} for more details. In our experiments, we use the polar decomposition as the retraction mapping. The domain $\Omega:= [0, 50]$ is discretized with $n$ equally spaced nodes. Since CMs problem \eqref{CMs} is not of finite-sum form, we only compare the results of the R-full-epoch algorithm (abbreviated as R-F-epo) with other algorithms. Table \ref{tab:cms_s1} presents the objective function value and CPU time of the algorithms on CMs problem with different values of $(n, r, \mu)$. Our algorithm outperforms other algorithms in most cases.

\begin{footnotesize}
\setlength{\tabcolsep}{2pt}
%\footnotesize
\centering
\begin{longtable}{|c|cccc|cccc|}
\caption{The results of our algorithm compared with other algorithms on CMs with random data. The row labeled ``Percentage'' indicates the ratio of the algorithm is dominant in CPU time among all compared algorithms.}\label{tab:cms_s1}\\
\hline
\multirow{2}{*}{$(n,r,\mu)$} & \multicolumn{4}{c|}{Objective function value} & \multicolumn{4}{c|}{CPU time} \\ 
        &  R-F-epo &  ManPG    & ManPG-A & R-sub  &  R-F-epo &  ManPG    & ManPG-A & R-sub \\ \hline
\endfirsthead
\hline
\multirow{2}{*}{$(n,r,\mu)$} & \multicolumn{4}{c|}{Objective function value} & \multicolumn{4}{c|}{CPU time} \\ 
        &  R-F-epo &  ManPG    & ManPG-A & R-sub  &  R-F-epo &  ManPG    & ManPG-A & R-sub \\ \hline
\endhead
\hline
\endfoot

(128,5,0.1) & 1.355e+0 & 1.355e+0 & 1.355e+0 & 1.360e+0 &  \textbf{0.25} & 0.55 & 0.48 & 2.15\\ \hline
      (128,5,0.1) & 2.356e+0 & 2.356e+0 & 2.356e+0 & 2.361e+0 &  1.06 & 0.21 & \textbf{0.14} & 2.38\\ \hline
      (128,5,0.2) & 4.097e+0 & 4.097e+0 & 4.097e+0 & 4.118e+0 &  0.58 & \textbf{0.01} & \textbf{0.01} & 2.32\\ \hline
      (128,5,0.3) & 5.661e+0 & 5.661e+0 & 5.661e+0 & 5.705e+0 &  0.64 & 0.08 & \textbf{0.05} & 2.41\\ \hline
      (128,10,0.1) & 2.937e+0 & 2.937e+0 & 2.937e+0 & 2.940e+0 &  1.09 & 1.41 & \textbf{0.95} & 4.07\\ \hline
      (128,10,0.1) & 4.815e+0 & 4.815e+0 & 4.815e+0 & 4.824e+0 &  1.50 & \textbf{1.44} & 2.20 & 6.12\\ \hline
      (128,10,0.2) & 8.206e+0 & 8.206e+0 & 8.206e+0 & 8.250e+0 &  \textbf{0.97} & 3.85 & 3.04 & 6.03\\ \hline
      (128,10,0.3) & 1.133e+1 & 1.133e+1 & 1.133e+1 & 1.142e+1 &  3.53 & \textbf{1.03} & 1.13 & 4.20\\ \hline
      (128,15,0.1) & 5.374e+0 & 5.375e+0 & 5.375e+0 & 5.376e+0 &  \textbf{0.18} & 2.85 & 3.13 & 6.50\\ \hline
      (128,15,0.1) & 8.012e+0 & 8.012e+0 & 8.012e+0 & 8.028e+0 &  \textbf{0.88} & 6.24 & 6.47 & 7.79\\ \hline
      (128,15,0.2) & 1.282e+1 & 1.282e+1 & 1.282e+1 & 1.287e+1 &  \textbf{0.49} & 3.12 & 3.11 & 6.72\\ \hline
      (128,15,0.3) & 1.726e+1 & 1.726e+1 & 1.726e+1 & 1.739e+1 &  \textbf{0.76} & 1.66 & 1.67 & 6.69\\ \hline
      (128,20,0.1) & 9.184e+0 & 9.184e+0 & 9.184e+0 & 9.184e+0 &  \textbf{0.74} & 17.58 & 24.86 & 16.65\\ \hline
      (128,20,0.1) & 1.253e+1 & 1.253e+1 & 1.253e+1 & 1.254e+1 &  \textbf{4.32} & 11.42 & 9.74 & 12.64\\ \hline
      (128,20,0.2) & 1.861e+1 & 1.861e+1 & 1.861e+1 & 1.869e+1 &  5.10 & 4.25 & \textbf{3.75} & 12.19\\ \hline
      (128,20,0.3) & 2.429e+1 & 2.429e+1 & 2.429e+1 & 2.447e+1 &  \textbf{4.77} & 11.97 & 10.86 & 12.72\\ \hline
      (128,30,0.1) & 2.270e+1 & 2.270e+1 & 2.269e+1 & 2.270e+1 &  \textbf{0.58} & 78.95 & 115.97 & 19.96\\ \hline
      (128,30,0.1) & 2.737e+1 & 2.737e+1 & 2.737e+1 & 2.738e+1 &  \textbf{0.65} & 59.77 & 81.07 & 29.15\\ \hline
      (128,30,0.2) & 3.575e+1 & 3.575e+1 & 3.575e+1 & 3.584e+1 &  \textbf{1.29} & 6.53 & 6.82 & 21.06\\ \hline
      (128,30,0.3) & 4.375e+1 & 4.375e+1 & 4.375e+1 & 4.395e+1 &  \textbf{0.58} & 3.69 & 3.33 & 20.62\\ \hline
      (256,5,0.1) & 1.788e+0 & 1.788e+0 & 1.788e+0 & 1.790e+0 &  \textbf{0.13} & 0.29 & \textbf{0.13} & 2.69\\ \hline
      (256,5,0.1) & 3.113e+0 & 3.113e+0 & 3.113e+0 & 3.124e+0 &  \textbf{0.38} & 1.13 & 0.88 & 2.66\\ \hline
      (256,5,0.2) & 5.416e+0 & 5.416e+0 & 5.416e+0 & 5.461e+0 &  0.59 & 0.24 & \textbf{0.23} & 2.70\\ \hline
      (256,5,0.3) & 7.489e+0 & 7.489e+0 & 7.489e+0 & 7.592e+0 &  0.31 & 0.03 & \textbf{0.02} & 3.35\\ \hline
      (256,10,0.1) & 3.747e+0 & 3.747e+0 & 3.747e+0 & 3.752e+0 &  \textbf{0.38} & 5.07 & 3.15 & 5.44\\ \hline
      (256,10,0.1) & 6.273e+0 & 6.273e+0 & 6.273e+0 & 6.291e+0 &  \textbf{0.86} & 2.86 & 3.53 & 6.35\\ \hline
      (256,10,0.2) & 1.084e+1 & 1.084e+1 & 1.084e+1 & 1.095e+1 &  \textbf{2.69} & 22.63 & 18.98 & 5.85\\ \hline
      (256,10,0.3) & 1.498e+1 & 1.498e+1 & 1.498e+1 & 1.519e+1 &  \textbf{0.59} & 1.37 & 1.68 & 5.64\\ \hline
      (256,15,0.1) & 6.522e+0 & 6.522e+0 & 6.522e+0 & 6.528e+0 &  \textbf{0.78} & 7.75 & 5.32 & 8.67\\ \hline
      (256,15,0.1) & 1.010e+1 & 1.010e+1 & 1.010e+1 & 1.013e+1 &  \textbf{1.05} & 4.17 & 3.62 & 9.42\\ \hline
      (256,15,0.2) & 1.659e+1 & 1.659e+1 & 1.659e+1 & 1.672e+1 &  \textbf{2.24} & 9.36 & 18.29 & 8.91\\ \hline
      (256,15,0.3) & 2.259e+1 & 2.259e+1 & 2.259e+1 & 2.294e+1 &  \textbf{2.72} & 14.02 & 25.48 & 8.42\\ \hline
      (256,20,0.1) & 1.068e+1 & 1.068e+1 & 1.068e+1 & 1.069e+1 &  \textbf{0.71} & 24.35 & 24.07 & 13.79\\ \hline
      (256,20,0.1) & 1.522e+1 & 1.522e+1 & 1.522e+1 & 1.526e+1 &  \textbf{1.46} & 12.17 & 14.15 & 15.48\\ \hline
      (256,20,0.2) & 2.350e+1 & 2.350e+1 & 2.350e+1 & 2.368e+1 &  \textbf{4.72} & 22.24 & 18.79 & 15.84\\ \hline
      (256,20,0.3) & 3.121e+1 & 3.121e+1 & 3.121e+1 & 3.160e+1 &  \textbf{0.91} & 14.40 & 15.64 & 13.94\\ \hline
      (256,30,0.1) & 2.509e+1 & 2.509e+1 & 2.509e+1 & 2.510e+1 &  \textbf{0.82} & 59.27 & 66.48 & 24.45\\ \hline
      (256,30,0.1) & 3.149e+1 & 3.149e+1 & 3.149e+1 & 3.152e+1 &  \textbf{0.41} & 35.48 & 41.17 & 24.90\\ \hline
      (256,30,0.2) & 4.308e+1 & 4.308e+1 & 4.308e+1 & 4.330e+1 &  \textbf{0.59} & 18.60 & 19.30 & 27.87\\ \hline
      (256,30,0.3) & 5.391e+1 & 5.391e+1 & 5.391e+1 & 5.447e+1 &  \textbf{0.62} & 31.60 & 26.37 & 24.85\\ \hline
      % (512,5,0.1) & 2.360e+0 & 2.360e+0 & 2.360e+0 & 2.369e+0 &  \textbf{0.48} & 2.54 & 1.02 & 3.87\\ \hline
      % (512,5,0.1) & 4.111e+0 & 4.111e+0 & 4.110e+0 & 4.148e+0 &  \textbf{0.75} & 5.19 & 4.20 & 3.76\\ \hline
      % (512,5,0.2) & 7.150e+0 & 7.150e+0 & 7.150e+0 & 7.308e+0 &  0.91 & 0.16 & \textbf{0.11} & 3.84\\ \hline
      % (512,5,0.3) & 9.889e+0 & 9.889e+0 & 9.889e+0 & 1.024e+1 &  \textbf{0.93} & 1.30 & 0.96 & 3.73\\ \hline
      % (512,10,0.1) & 4.829e+0 & 4.828e+0 & 4.827e+0 & 4.857e+0 &  \textbf{2.72} & 13.38 & 8.05 & 8.33\\ \hline
      % (512,10,0.1) & 8.244e+0 & 8.244e+0 & 8.230e+0 & 8.323e+0 &  \textbf{3.07} & 12.93 & 18.91 & 7.66\\ \hline
      % (512,10,0.2) & 1.431e+1 & 1.431e+1 & 1.430e+1 & 1.461e+1 &  \textbf{1.84} & 16.08 & 14.91 & 7.95\\ \hline
      % (512,10,0.3) & 1.978e+1 & 1.978e+1 & 1.978e+1 & 2.049e+1 &  \textbf{2.70} & 9.76 & 6.34 & 8.04\\ \hline
      % (512,15,0.1) & 8.052e+0 & 8.053e+0 & 8.053e+0 & 8.074e+0 &  \textbf{2.69} & 456.86 & 13.08 & 16.04\\ \hline
Percentage & & & & & \textbf{77.5\%} & 7.5\% &17.50\% & 0 \\ \hline

\end{longtable}
\end{footnotesize}

\section{Conclusions}

In this paper, we present a novel family of Riemannian gradient-based methods, namely Riemannian smoothing gradient and Riemannian smoothing stochastic gradient, designed to identify generalized $\epsilon$-stationarity points for a class of nonconvex and nonsmooth problems on compact Riemannian submanifold embedded in Euclidean space.
We introduce the concept of generalized $\epsilon$-stationarity, which extends the notion of $\epsilon$-stationarity. By leveraging the Moreau envelope with a sequence of decreasing smoothing parameters, we transform the problem into a smoothed formulation in Euclidean space. We establish the smoothness of this problem on the Riemannian compact submanifold.
We demonstrate that our proposed algorithms exhibit iteration complexities of $\mathcal{O}(\epsilon^{-3})$ and $\mathcal{O}(\epsilon^{-5})$ to achieve a generalized $\epsilon$-stationarity, respectively. Notably, our algorithms outperform other existing methods in numerical experiments conducted on both the SPCA and CMs problems.

\section*{Declarations}
%\subsection*{Funding}
This research was supported by  the Natural Science Foundation of China with grant 12071398,   the Natural Science Foundation of Hunan Province with grant 2020JJ4567,  and the Key Scientific Research Found of Hunan Education Department with grants 20A097 and 18A351.
%\subsection*{Competing Interests}
The authors have no relevant financial or non-financial interests to disclose.

% \begin{itemize}
% \item Funding
% \item Conflict of interest/Competing interests (check journal-specific guidelines for which heading to use)
% \item Ethics approval 
% \item Consent to participate
% \item Consent for publication
% \item Availability of data and materials
% \item Code availability 
% \item Authors' contributions
% \end{itemize}

% \noindent
% If any of the sections are not relevant to your manuscript, please include the heading and write `Not applicable' for that section. 

%%===================================================%%
%% For presentation purpose, we have included        %%
%% \bigskip command. please ignore this.             %%
%%===================================================%%
% \bigskip
% \begin{flushleft}%
% Editorial Policies for:

% \bigskip\noindent
% Springer journals and proceedings: \url{https://www.springer.com/gp/editorial-policies}

% \bigskip\noindent
% Nature Portfolio journals: \url{https://www.nature.com/nature-research/editorial-policies}

% \bigskip\noindent
% \textit{Scientific Reports}: \url{https://www.nature.com/srep/journal-policies/editorial-policies}

% \bigskip\noindent
% BMC journals: \url{https://www.biomedcentral.com/getpublished/editorial-policies}
% \end{flushleft}

%%===========================================================================================%%
%% If you are submitting to one of the Nature Portfolio journals, using the eJP submission   %%
%% system, please include the references within the manuscript file itself. You may do this  %%
%% by copying the reference list from your .bbl file, paste it into the main manuscript .tex %%
%% file, and delete the associated \verb+\bibliography+ commands.                            %%
%%===========================================================================================%%

\bibliography{sn-bibliography}% common bib file
%% if required, the content of .bbl file can be included here once bbl is generated
%%\input sn-article.bbl

%% Default %%
%%\input sn-sample-bib.tex%

\end{document}